\documentclass[a4paper,10pt]{svjour3}

\usepackage[T1]{fontenc}
\usepackage[utf8]{inputenc}

\usepackage{amsmath}
\usepackage{amsfonts}
\usepackage{amssymb}

\usepackage{cases}
\usepackage{todonotes}
\usepackage{xspace}

\newtheorem{assumption}{Assumption}

\smartqed

\usepackage{graphicx}
\usepackage{pgfplots}
\pgfplotsset{compat=1.3}
\tikzset{
    png export/.style={
        external/system call/.add={}%
        {; test -f "\image.pdf" && convert -density 300 -transparent white "\image.pdf" "\image.png" || true}
    }
}
\tikzset{
    eps export/.style={
        external/system call/.add={}%
        {; test -f "\image.pdf" && pdftops -eps "\image.pdf" "\image.eps" || true}
    }
}
\tikzset{png export}
\tikzset{eps export}
\RequirePackage[hang,bf]{caption}
\newlength\figureheight
\newlength\figurewidth
\newcommand{\includetikz}[3]{%
	\ifx\disabletikzimages\undefined%
		\setlength\figurewidth{#1}%
		\setlength\figureheight{#2}%
		\input{#3}%
	\else%
		\framebox[#1]{\parbox[c][#2]{#1}{%
			\centering $<$ tikz images disabled $>$\\%
		}}%
	\fi
}

\newcommand{\enabletikz}[0]{\let\disabletikzimages\undefined}

\usepackage{hyperref}

\title{Convergence analysis of an explicit splitting method for laser plasma interaction simulations}
\titlerunning{Triple splitting for laser plasma interactions -- convergence analysis}
\author{Georg Jansing \and Achim Sch\"adle}

\institute{
Georg Jansing
\at Mathematisches Insititut, Heinrich-Heine Universit\"at, Universitätsstraße 1, 40225 D\"usseldorf, Germany\\
\email{georg.jansing@hhu.de}
\and
Achim Sch\"adle
\at Mathematisches Insititut, Heinrich-Heine Universit\"at, Universitätsstraße 1, 40225 D\"usseldorf, Germany\\
\email{schaedle@hhu.de}
}

\newcommand{\mat}[1]{\left[ ~ \begin{matrix} #1 \end{matrix} ~ \right]}
\DeclareMathOperator{\dd}{d}
\DeclareMathOperator{\Id}{Id}
\DeclareMathOperator{\sinc}{sinc}
\DeclareMathOperator{\cosc}{cosc}

\newcommand{\vv}{\ensuremath{\mathbf{v}}\xspace}
\newcommand{\pp}{\ensuremath{\mathbf{p}}\xspace}
\newcommand{\EE}{\ensuremath{\mathbf{e}}\xspace}
\newcommand{\BB}{\ensuremath{\mathbf{b}}\xspace}

\newcommand{\EEn}{\EE_n}
\newcommand{\BBn}{\BB_n}

\newcommand{\ppn}{\pp_n}

\newcommand{\EEnp}{\EE_{n+1}}
\newcommand{\BBnp}{\BB_{n+1}}
\newcommand{\ppnp}{\pp_{n+1}}

\newcommand{\BBnph}{\BB_{n+\frac 1 2}}
\newcommand{\EEnplus}{\EE_n^+}
\newcommand{\EEnpminus}{\EE_{n+1}^-}

\newcommand{\GG}{\ensuremath{\mathbf{G}}}
\newcommand{\HH}{\ensuremath{\mathcal{H}}}
\newcommand{\CCB}{\ensuremath{\mathbf{C}_B}}
\newcommand{\CCE}{\ensuremath{\mathbf{C}_E}}
\newcommand{\OOmega}{\mathbf{\Omega}}
\renewcommand{\~}[1]{\widetilde{#1}}

\newcommand{\RR}{\ensuremath{\mathbb{R}}}
\newcommand{\CC}{\ensuremath{\mathbb{C}}}

\definecolor{blue}{rgb}{0,0,1}
\definecolor{red}{rgb}{0.8,0.3,0.08}
\definecolor{green}{rgb}{0.2,0.5,0.02}

\begin{document}
\maketitle
\begin{abstract}
Convergence of a triple splitting method originally proposed
in~\cite{TuePLH10,Lil10} for the solution of a simple
Vlasov-Maxwell system, that describes laser plasma interactions with
overdense plasmas, is analyzed. For classical explicit integrators it is
the large density parameter that would impose a restriction on the time
step size to make the integration stable. 
The triple splitting method contains an exponential integrator in
its central component and was specifically designed for systems that
describe laser plasma interactions and overcomes this restriction.  We
rigorously analyze a slightly generalized version of the original
method. This analysis enables us to identify modifications of the original scheme, 
such that a second order convergent scheme is obtained.
\end{abstract}


\section{Introduction} \label{sec:intro}

We consider the numerical solution of a simplified Vlasov-Maxwell
system of equations, describing laser plasma interactions with an
overdense plasma. After discretizing in space for a fixed spatial
grid parameter $\Delta x$ a system of ordinary differential equations
is obtained. The situation we wish to consider now is slightly unusual
as it is the overdense plasma and not the space discretization that
gives rise to fast oscillations in the solution. And hence it would be
the plasma frequency that would impose a step size restriction in
explicit Runge-Kutta or multistep methods. To overcome the
restriction on the time step size due the plasma frequency a triple
splitting method with filter functions was introduced
by Liljo and T\"uckmantel, Pukhov, Liljo and Hochbruck in~\cite{Lil10,TuePLH10} 
for this model problem. An astute choice
of filter functions results in a method that shows excellent behavior
in numerical experiments. Numerical experiments in~\cite{TuePLH10}
indicate convergence of second order in the time step size $\tau$
independent of the plasma density $\rho$. A more detailed experiment, 
which is reported in Section~\ref{sec:numexp}, reveals that the method
from~\cite{TuePLH10} is not second order in $\tau$ independent of the
plasma density $\rho$, but is merely stable.

By our convergence analysis of the triple splitting we are able to
formulate conditions on the filter functions to obtain second order
convergence in $\tau$ independent of the plasma density $\rho$.  These
conditions can be fulfilled by slightly modifying the choice of the filter
functions originally proposed in~\cite{Lil10,TuePLH10}.

As the triple splitting is an explicit integrator the method certainly 
can not be expected to be convergent uniformly in $\Delta x \to 0$. 
Thus our aim here is to prove convergence independent of the large plasma density $\rho$ 
but not independent of the spatial discretization parameter $\Delta x$.  

In a nutshell the plan for the convergence proof is as follows: 
The triple splitting for the impulse of the plasma density $\pp$, the
electric field $\EE$ and the magnetic flux $\BB$ will be reformulated
as a two step method for $\EE$ only with some sort of ``natural''
filter. Perturbing the initial values this reformulation allows to
estimate the error in $\EE$ using a result from Hairer, Lubich and Wanner~\cite[Theorem
XIII.4.1]{HaiLW06}. We then show that the perturbation in the initial
values is small enough, such that by a stability argument convergence
for $\EE$ is obtained. The estimates for the magnetic flux $\BB$ and
the impulse $\pp$ are obtained by a judicious combination of ideas
borrowed from Grimm and Hochbruck~\cite{GriH06} with trigonometric identities.
The present paper is based on the first part of the PhD thesis~\cite{Jan15}.

\section{Physical problem and spatial discretization}\label{sec:introPhysics}

Consider the propagation of a short laser pulse in vacuum targeted
at a plasma around a thin foil. The electric field $\EE$ and the magnetic flux
$\BB$ describing the laser are governed by Maxwell's equations. In
this simple model the plasma is modeled as a fluid by the electron
number density $\rho$ (number of electrons per volume) and the
probability density function of the impulses of the electrons $\pp$.
The laser plasma interactions with an
overdense plasma ($\rho \gg 1$) and a linear response of the 
plasma to the laser is modeled by
\begin{subequations}
\label{eq:maxwellsEquationsLaserReflect}
\begin{align}
	\partial_t \pp & = \EE, & x\in B, t>0, \label{eq:maxwellsEquationsLaserReflectp}\\
	\partial_t \EE & = \nabla \times \BB - f^2 \rho \pp, &x\in B, t>0, \\
	\partial_t \BB & = -\nabla \times \EE, &x\in B, t>0.
\end{align}
\end{subequations}
Here $f = 2 \pi e$, where $e$, the electron charge, is a constant. $B$
is the computational domain, a box, containing the plasma and the support
of the initial values.  The vacuum
permittivity (electric constant) and permeability (magnetic constant)
are set to $1$. In our simplified model plasma only oscillates locally,
thus its impulses $\pp$ also oscillate, but the density $\rho$, remains
constant. There are two further
essential assumptions. We assume that the electrons move slowly, such that
relativistic effects can be neglected, i.e.  the velocity field of the
plasma $\vv$ is proportional to the impulse $\pp$.  Secondly we
neglect the magnetic Lorentz force $\vv \times \BB$. These rather
restrictive assumptions make the
model~\eqref{eq:maxwellsEquationsLaserReflect} linear.  A more
detailed derivation of the model may be found
in~\cite{TuePLH10,Tue13}.

Equation~\eqref{eq:maxwellsEquationsLaserReflect} has to be
supplemented with boundary conditions and initial values. The theory
developed below applies to the case of perfect magnetic conductor
(PMC), perfect electric conductor (PEC) or periodic boundary
conditions, which guarantee that the ``curl curl'' operator is
self-adjoint~\cite{HipKT12}.

As we only discuss the convergence of the semi-discrete problem in
the following, the solution of the spatially discretized equations will
again be denoted by $\pp$, $\EE$ and $\BB$. Discretizing in space with the
Yee scheme or curl-conforming finite elements we denote by $\CCE$ and
$\CCB$ discrete versions of the ``curl'' applied to $\EE$
and $\BB$ respectively. Note that these curl-operators are allowed to be
different and should be different. The electric field $\EE$ can conveniently be 
interpreted as a differential $1$-form, then $\CCE$ is a discrete version
of ``curl''. Whereas in this context $\BB$ has to be interpreted as differential
$2$-form, such that $\CCB$ is as discrete version of ``*curl*'', where
* is the Hodge operator~\cite{Hip02}.

The multiplication with $f^2\rho$ is discretized by a matrix
$\OOmega^2$. In case of the Yee scheme $\OOmega^2$ is a diagonal matrix. 
In case one uses curl-conforming finite elements,
$\OOmega^2$ is a positive semidefinite matrix and mass matrices arise
on the right hand
side of~\eqref{eq:maxwellsEquationsLaserReflect}. In what follows we will assume
that $\OOmega^2$ is a diagonal matrix with only one positive eigenvalue. 
Generalizations to a non-diagonal but symmetric positive semidefinite 
discretization $\OOmega^2$ of the multiplication 
operator will be discussed in Section~\ref{sec:multifrequency}.

If space is scaled to the wave number and time to the laser frequency
the spatially discretized equations are
\begin{subequations}
\label{eq:maxwellsEquationsDiscrete}
\begin{align}
	\label{eq:maxwellsEquationsDiscrete_p}
	\partial_t \pp & =  \EE, &t>0,\\
	\label{eq:maxwellsEquationsDiscrete_E}
	\partial_t \EE & =  \CCB \BB - \OOmega^2 \pp, &t>0, \\
	\label{eq:maxwellsEquationsDiscrete_B}
	\partial_t \BB & = -\CCE \EE,&t>0.
\end{align}
\end{subequations}
Assuming for the moment that $\rho$ vanishes,  
a right traveling pulse with width parameter $\sigma_0 = 10$ and wavelength $1$ 
solving~\eqref{eq:maxwellsEquationsLaserReflect} is given by
\begin{equation}
  \label{eq:laserPulse}
  \begin{aligned}
    &\EE_y = \BB_z = a_0 \exp\left(-\frac{(2\pi[(x - \bar{x}) - t])^2}{2
        \sigma_0^2}\right) \cos(2\pi[(x - \bar{x}) - t]), \\
    &\EE_x \equiv \EE_z \equiv \BB_x  \equiv \BB_y \equiv 0.
  \end{aligned}
\end{equation}
If we set $t = 0$, $\bar{x} = 10$ and $a_0 =1$ in~\eqref{eq:laserPulse} initial values 
for a pulse centered at $10$ with amplitude $1$ are obtained.

The plasma is located away from the initial
location of the pulse by choosing $\rho = 10^{8} $
for $x \in [20, 21]$ and $\rho = 0$ elsewhere. This 
leads to a total reflection of the laser pulse on the edge of the plasma. 
Figure \ref{fig:laserReflectSolution}
shows different snap shots of the simulation.

\begin{figure}[tbh]
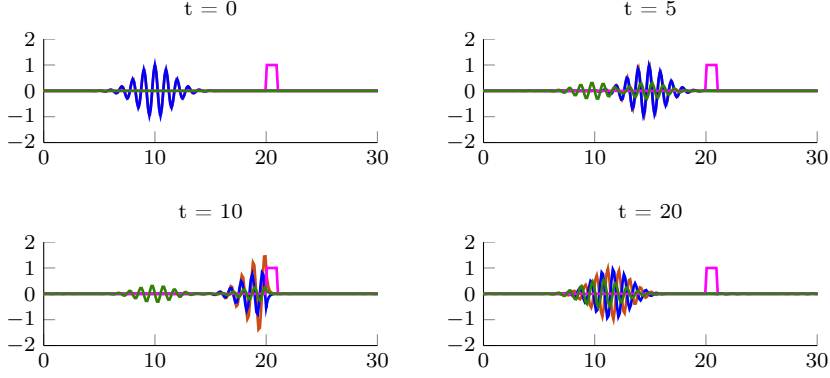

	\centering
	\begin{minipage}[b]{0.45\textwidth}
		\includetikz{0.8\textwidth}{0.25\textwidth}{laserReflect_T_0.tikz}
	\end{minipage} ~
	\begin{minipage}[b]{0.45\textwidth}
		\includetikz{0.8\textwidth}{0.25\textwidth}{laserReflect_T_5.tikz}
	\end{minipage} \\[1em]
	\begin{minipage}[b]{0.45\textwidth}
		\includetikz{0.8\textwidth}{0.25\textwidth}{laserReflect_T_10.tikz}
	\end{minipage} ~
	\begin{minipage}[b]{0.45\textwidth}
		\includetikz{0.8\textwidth}{0.25\textwidth}{laserReflect_T_20.tikz}
	\end{minipage} \\[1em]
	\caption{Solution of \eqref{eq:maxwellsEquationsLaserReflect}
          with the laser pulse from \eqref{eq:laserPulse} as initial value.
          Impulse $\pp_y$ (green), electric field $\EE_y$ (red), magnetic flux
          $\BB_z$ (blue) and electron density $\rho$ (magenta) scaled to one.
          First graphic ($t=0$): Initial data, second ($t=5$): propagation of
          the pulse in vacuum, third ($t=10$): total reflection at the foil,
          fourth ($t=20$): back propagation}
	\label{fig:laserReflectSolution}
      \end{figure}

\section{Numerical scheme and filter functions}
\label{sec:numericalScheme}

To solve the spatially discretized equations~\eqref{eq:maxwellsEquationsDiscrete} 
we use the triple splitting method proposed by Liljo and T\"uckmantel, Pukhov, Liljo and Hochbruck in \cite{Lil10,TuePLH10}.
To this end the right hand side is split into three terms
\begin{alignat*}{3}
  \partial_t \mat{ \pp \\ \EE \\ \BB }
  & =
    \mat{ 0 \\ 0 \\ -\CCE \EE }
  & + &
  \mat{ 0 \\ \CCB \BB \\ 0 }
  & + &
  \mat { 0 & 1 \\ - \OOmega^2 & 0 \\ & & 0 } \mat{ \pp \\ \EE \\ \BB }
  := f_1 + f_2 + f_3.
\end{alignat*}

The fully discrete scheme is a symmetric triple splitting obtained by
taking the exact flows of the split equations (i.e. with only one
$f_i$ as right hand side) as propagators. As already observed in
\cite{Lil10,TuePLH10} due to resonances this is not sufficient for convergence independent of
$\rho$.  To introduce filter functions is a widely used mean
to avoid resonance effects, see e.g.~\cite{Garcia-Archilla1998,HocL99,Hairer2001,GriH06,HaiLW06}.
We follow~\cite{Lil10,TuePLH10}, introduce filter functions symmetrically and obtain the following numerical scheme
\begin{subequations}
\label{eq:timeIntegrationScheme}
\begin{align}
 	\label{eq:maxwellSplittingSchemeWithFilterpTildeBnph}
	\BBnph & = \BBn - \tfrac \tau 2 \textcolor{blue}{\psi_B(\tfrac \tau 2 \OOmega)} \CCE \textcolor{green}{\phi_E(\tfrac \tau 2 \OOmega)} \EEn, \\
	\label{eq:maxwellSplittingSchemeWithFilterpTildeEnplus}
	\EEnplus & = \EEn + \tfrac \tau 2 \textcolor{green}{\psi_E(\tfrac \tau 2 \OOmega)} \CCB \textcolor{blue}{\phi_B(\tfrac \tau 2 \OOmega)} \BBnph, \\
	\label{eq:maxwellSplittingSchemeWithFilterpTildeOsci}
	\mat{\ppnp \\ \EEnpminus} & = \mat{
		\cos(\tau \OOmega)          & \tau \sinc(\tau \OOmega) \\
		-\OOmega \sin(\tau \OOmega) & \cos(\tau \OOmega) \\
	} \mat{ \ppn \\ \EEnplus }, \\
	\label{eq:maxwellSplittingSchemeWithFilterpTildeEnp}
	\EEnp & = \EEnpminus + \tfrac \tau 2 \textcolor{green}{\psi_E(\tfrac \tau 2 \OOmega)} \CCB \textcolor{blue}{\phi_B(\tfrac \tau 2 \OOmega)} \BBnph, \\
	\label{eq:maxwellSplittingSchemeWithFilterpTildeBnp}
	\BBnp & = \BBnph - \tfrac \tau 2 \textcolor{blue}{\psi_B(\tfrac \tau 2 \OOmega)} \CCE \textcolor{green}{\phi_E(\tfrac \tau 2 \OOmega)} \EEnp.
\end{align}
\end{subequations}
For $i \in \{ E, B \}$ we require $\psi_i$, $\phi_i$ to be even, analytic functions such that
$\psi_i(z), \phi_i(z) \to 1$ for $z \to 0$.

In the following section we state assumptions on the physical data and
the spatial discretization that are necessary for the convergence
proof.

\section{Assumptions}
\label{sec:assumptions}

The following assumptions are not too restrictive from a theoretical physics point of view.
They are fulfilled for example in the situation considered in~\cite{Lil10,TuePLH10}
simulating the reflection of a laser pulse by a plasma.

\begin{assumption}
  \label{ass:physicalAndDiscretizationAssumptions}
  We assume that 
  \begin{enumerate}
  \item[(i)] the product $\CCB \CCE =: -\GG = -\GG^T$ is symmetric, positive semidefinite,
  \item[(ii)] $\OOmega$ is a diagonal matrix given by
    \begin{equation}
      \label{eq:omegaMatrixStructure}
      \OOmega = \mat{ 0 & 0 \\ 0 & \~\omega \Id},
      \qquad
      \~\omega \gg 1
    \end{equation}
    and
  \item[(iii)] $~$\\[-7ex]
    \begin{equation}
      \label{eq:curlBounded}
      \| \CCE \| \leq C_c,
      \ \text{and} \
      \| \CCB \| \leq C_c,
      \quad
      \text{such that}
      \quad
      \| \GG \| \leq C_g := C_c^2
    \end{equation}
    for a constant $C_c$ independent of $\~\omega$.
  \end{enumerate}
\end{assumption}

The symmetry and negative semi-definiteness of $\GG$ comes quite
natural, provided that the continuous ``curl curl'' operator is
self-adjoint and positive semi-definite, which is the case for
appropriate boundary conditions such as perfect electric conductor
(PEC), perfect magnetic conductor (PMC) or periodic boundary
conditions, see~\cite{HipKT12}.
Condition~\eqref{eq:omegaMatrixStructure} implies that the matrix
$\OOmega$ has only one (large) non-zero eigenvalue $\~\omega > 0$.  In
our case it is given by the density parameter, i.e. $\~\omega = f
\sqrt{\rho}$. This is an essential restriction, which is only needed
in the proof of Theorem~\ref{thm:errorE}.  A modification of our
proof, that only requires $\OOmega$ to be symmetric positive
semi-definite, is given by Buchholz and Hochbruck in~\cite{BucH15} and 
will be discussed in Section~\ref{sec:multifrequency}. Estimates~\eqref{eq:curlBounded}
imply that we do not obtain error bounds uniformly in the spatial
discretization parameter, e.g. the mesh width. As mentioned already
in the introduction this would indeed be impossible for our
integration scheme since it reduces to the Störmer-Verlet method in
case of no material ($\rho \equiv 0$), which is known to be
conditionally stable only.
Assumption~\ref{ass:physicalAndDiscretizationAssumptions} is for
example satisfied if the curl operators with periodic boundary
conditions are discretized using a Yee-scheme and a step function
$\rho$ is evaluated point-wise.

Additionally we need some bounds on the initial data $\EE_0, \BB_0$ and $\pp_0$ to obtain
stable solutions as discussed in Section~\ref{sec:reformulation}:
\begin{assumption}
\label{ass:initialData}
We assume that 
\begin{equation}
  \label{eq:initialValueConditions}
  \| \OOmega \EE_0 \|^2 \leq \tfrac 2 3 H_0,
  \quad
  \| \CCB \BB_0 \|^2 \leq \tfrac 1 3 \min \lbrace 1, \tfrac{4}{C_4^2} \rbrace H_0,
  \quad
  \| \OOmega^2 \pp_0 \|^2 \leq \tfrac 1 3 H_0,
\end{equation}
\begin{equation}
  \label{eq:hamiltonianConservationCondition}
  -\left< \EE_0, \GG \EE_0 \right> = \| \CCE \EE_0 \|^2 \leq 2 H_0,
  \qquad
  \| \EE_0 \|^2 \leq H_0, 
\end{equation}
\begin{equation}
  \| \BB_0 \|^2 \leq H_0 \quad \mbox{and} \quad
  \label{eq:stabilityAssumptionInitialValue}
  \| \OOmega^2 \CCB \BB_0 \|^2 \leq H_0,
\end{equation}
for constants $H_0$ and $C_4$ independent of $\~\omega$.
\end{assumption}


A bound with respect to multiplication with $\OOmega$ implies, that the initial
data are sufficiently far away from the plasma, such that the product
of the field strength and the density is bounded independent of the
density. Bounds with respect to  multiplications with $\CCB$ or $\CCE$ 
maybe seen as smoothness conditions for the initial data.

\section{Main theorem}
\label{sec:filterConditionsAndconvergenceResults}
For our convergence result we need conditions on the filter functions, which we collect below.
As in~\cite{TuePLH10} we require
\begin{align}
  \label{eq:filterAssumptionsBequiv1}
  \phi_B \equiv \psi_B \equiv 1.
\end{align}
The following bounds for the filter functions are required for second order convergence of the 
scheme~\eqref{eq:timeIntegrationScheme}
\begin{subequations}
\label{eq:filterAssumptionsE}
\begin{align}
  \label{eq:filterAssumptionPsi}
  | (\cos(z) + 1)\psi_E(\tfrac 1 2 z) |
  & \leq C_1 \sinc^2(\tfrac 1 2 z), \\
  \label{eq:filterAssumptionPhi}
  | \phi_E(\tfrac 1 2 z) |
  & \leq C_2 |\sinc(\tfrac 1 2 z)|, \\
  \label{eq:filterAssumptionPsiTimesPhi}
  | (\cos(z) + 1) \psi_E(\tfrac 1 2 z) \phi_E(\tfrac 1 2 z) |
  & \leq C_3 |\sinc(z)|, \\
  \label{eq:filterAssumptionPsi_Sinc}
  | (\cos(z) + 1)\psi_E(\tfrac 1 2 z) |
  & \leq C_4 |\sinc(z)|, \\
  \label{eq:filterAssumptionSincMinusPsi}
  \left| \sinc(z) - \tfrac 1 2 (\cos(z) + 1) \psi_E(\tfrac 1 2 z) \right|
  & \leq C_5 z^2 | \sinc(z) |, \\
  \label{eq:filterAssumptionSincMinusPhi}
  | \sinc(z) - \phi_E(\tfrac 1 2 z) |
  & \leq C_6 | z \sin(\tfrac 1 2 z) |, \\
  \label{eq:filterAssumptionPsiBounded}
  | \psi_E(z) |
  & \leq C_7
  \intertext{and}
  \label{eq:filterAssumptionSinc2MinusSincPhi}
  | \sinc^2(\tfrac 1 2 z) - \sinc(z) \phi_E(\tfrac 1 2 z) |
  & \leq C_8 \sin^2(\tfrac 1 2 z).
\end{align}
\end{subequations}
With these conditions we obtain our main result
\begin{theorem}
  \label{thm:errorEstimateFinal}
  Let $\CCB, \CCE$ and $\OOmega$ be such that Assumption
  \ref{ass:physicalAndDiscretizationAssumptions} is fulfilled.  Consider
  the numerical solution of the
  system~\eqref{eq:maxwellsEquationsDiscrete} by the splitting
  method~\eqref{eq:timeIntegrationScheme} with time step size $\tau$
  satisfying $\tau \leq \tau_0$, for sufficiently small $\tau_0$
  independent of $\~\omega$ with $\tau \~\omega \geq c_0 > 0$.  If the
  initial values satisfy conditions \eqref{eq:initialValueConditions} to
  \eqref{eq:stabilityAssumptionInitialValue} with a constant $H_0$
  independent of $\~\omega$ and the filter functions
  satisfy~\eqref{eq:filterAssumptionPsi}-\eqref{eq:filterAssumptionSinc2MinusSincPhi}
  then for $t_n := t_0 + n \tau \leq T$ we obtain the following second
  order estimates for the errors
  \begin{equation}
    \label{eq:errorBoundNew}
    \| \ppn - \pp(t_n) \| \leq C \tau^2,
    \quad
    \| \EEn - \EE(t_n) \| \leq C \tau^2,
    \quad
    \| \BBn - \BB(t_n) \| \leq C \tau^2.
  \end{equation}
  The constant $C$ is independent of ~$\~\omega$, $\tau$, $n$ or derivatives of the solution, but depends on $(T - t_0)$ and the constants $C_1, ..., C_8$ and $C_c$ in \eqref{eq:curlBounded} and $H_0$.
\end{theorem}

The assumption $\tau \~\omega \geq c_0 > 0$ is the interesting case and is no restriction as 
otherwise we are in the case of classical convergence analysis.
The proof is given in Section \ref{sec:proof}.

\paragraph{Choice of the filter functions}
Tückmantel et. al.~\cite{TuePLH10} propose the choice
\begin{equation}
	\label{eq:filterChoiceJalo}
	\phi_E = \psi_E = z \mapsto \sinc(z),
	\qquad
	\phi_B \equiv \psi_B \equiv 1,
\end{equation}
which satisfies conditions~\eqref{eq:filterAssumptionPhi}
to~\eqref{eq:filterAssumptionSinc2MinusSincPhi}.  It does not obey
condition~\eqref{eq:filterAssumptionPsi} but only the weaker estimate
$| (\cos(z) + 1)\psi_E(\tfrac 1 2 z) | \leq C_0 \sinc(\tfrac 1 2
z)$. Detailed numerical tests in Section~\ref{sec:numexp} for this
choice reveal sharp resonances a even multiples of $\pi/\omega$, which
where not observed in~\cite{TuePLH10}.

We propose the new choice
\begin{equation}
  \label{eq:filterChoiceNew}
  \phi_E = z \mapsto \sinc(z),
  \qquad
  \psi_E = z \mapsto \sinc^2(z),
  \qquad
  \phi_B \equiv \psi_B \equiv 1,
\end{equation}
which also statisfies the first filter condition \eqref{eq:filterAssumptionPsi} 
and thus by Theorem~\ref{thm:errorEstimateFinal} results in second order error bounds.
The detailed proof that the filter functions~\eqref{eq:filterChoiceNew} meet all of 
conditions~\eqref{eq:filterAssumptionsE} can be found in \cite[Section 4.12]{Jan15}.

These two choices of filter functions are used on the numerical experiment 
in Section~\ref{sec:numexp}. Figure~\ref{fig:ts} there shows the error
for the scheme~\eqref{eq:timeIntegrationScheme}, without filter (None), with the filter
choice \eqref{eq:filterChoiceNew} (New), which yields a
second order scheme uniformly in $\omega$, and with the filter choice~\eqref{eq:filterChoiceJalo} (Orig)
which violates \eqref{eq:filterAssumptionPsi} and shows 
sharp resonances and a breakdown of the method if $\tau\~\omega$ is
close to even multiples of $\pi$.

\begin{remark}
  \label{rem.1}
  Theorem 4.19 of~\cite{Jan15} claims that for the filter
  choice~\eqref{eq:filterChoiceJalo} one obtains
  convergence of order one in $\tau$ independent of $\omega$.
  This however is shown to be wrong by our numerical tests in Section~\ref{sec:numexp}. 
  Theorem 4.19 of~\cite{Jan15} was derived in an analogous way to the second order
  result presented below. It is based on a supplementary first order
  convergence result for the two step method~\eqref{eq:maxwellSplittingSchemeMultistepFormulation} 
  given in~\cite[Theorem XIII.4.1]{HaiLW06} for the weakened filter assumption  
  $| (\cos(z) + 1)\psi_E(\tfrac 1 2 z) | \leq C_0 \sinc(\tfrac 1 2 z)$ 
  replacing~\eqref{eq:filterAssumptionPsi}, see also Remark~\ref{rem.2} below. 
\end{remark}

\section{Proof of Theorem~\ref{thm:errorEstimateFinal}}
\label{sec:proof}
The proof is divided into four steps. First we reformulate the
scheme~\eqref{eq:timeIntegrationScheme} as a two step method for the
electric field $\EE$ only. With this reformulation we can apply an already
known error estimate to control the error in the electric field after
modifying the intitial values. Based on the error bound for $\EE$
error bounds for $\BB$ and $\pp$ are obtained. A more detailed proof
can be found in~\cite[Chapter 4]{Jan15}.

\subsection{Reformulation}
\label{sec:reformulation}

From equation~\eqref{eq:maxwellsEquationsDiscrete} one obtains an equation for electric field
\begin{subnumcases}{\label{eq:modFourierModelEqMaxwell}}
	\label{eq:modFourierModelEqMaxwellODE}%
	\partial_{tt} \EE(t) = -\OOmega^2 \EE(t) + \GG \EE(t), \\
	\label{eq:modFourierModelEqMaxwellIV}%
	\EE(t_0) = \EE_0,
	\quad
	\partial_t \EE(t_0) = \CCB \BB(t_0) - \OOmega^2 \pp(t_0) := \dot \EE_0
\end{subnumcases}
with Hamiltonian
\begin{equation}
	\label{eq:maxwellHamiltonFunctionEDiscrete}
	\HH(\EE, \mathbf{f})
	= \tfrac 1 2 \| \mathbf{f} \|^2 + \tfrac 1 2 \| \OOmega \EE \|^2 - \tfrac 1 2 \left< \EE, \GG \EE \right>
	= \tfrac 1 2 \| \mathbf{f} \|^2 + \tfrac 1 2 \| \OOmega \EE \|^2 + \tfrac 1 2 \| \CCE \EE \|^2.
\end{equation}
From Assumption \ref{ass:initialData} we deduce the stability estimates
\begin{align}
	\label{eq:HamiltonBounded}
	\HH(\EE(t), \partial_t \EE(t))
	& \leq 2 H_0, \\
	\label{eq:EFieldBounded}
	\| \EE(t) \|
	& \leq (1 + 2 (T - t_0)) \sqrt{H_0}, \\
	\label{eq:BFieldBounded}
	\| \BB(t) \|
	& \leq (1 + 2 (T - t_0)) \sqrt{H_0}
\end{align}
for $t_0 \leq t \leq T$. The latter two can be obtained by expressing $\EE(t)$ and $\BB(t)$ with the fundamental theorem of calculus and exploiting that the integrands $\partial_t \EE$ and $-\CCE\EE$ are both bounded by the Hamiltonian.
The variation of constants formula gives the following representation of the solution $\EE$ of~\eqref{eq:modFourierModelEqMaxwell} starting from $t'$ with initial data $\EE(t')$ and $\partial_t \EE(t')$
\begin{align}
  \notag
  \EE(t)
  = &\cos((t - t') \OOmega) \EE(t') + (t - t') \sinc((t - t') \OOmega) \partial_t \EE(t') \\
  \label{eq:EfieldFormula}
  &  + (t - t') \int_0^1 (t - t') (1 - \xi) \sinc((t - t') (1 - \xi) \OOmega) \GG \EE(t' (1 - \xi) + t \xi) \dd\!\xi \\
  \intertext{and similar for $\partial_t \EE$}
  \notag
  \partial_t \EE(t)
  =& -\OOmega \sin((t - t') \OOmega) \EE(t') + \cos((t - t') \OOmega) \partial_t \EE(t') \\
  \label{eq:dotEfieldFormula}
  & + (t - t') \int_0^1 \cos((t - t') (1 - \xi) \OOmega) \GG \EE(t' (1 - \xi) + t \xi) \dd\!\xi.
\end{align}
For an $\EE$-only formulation for the numerical scheme, we use~\eqref{eq:maxwellSplittingSchemeWithFilterpTildeBnph},
\eqref{eq:maxwellSplittingSchemeWithFilterpTildeEnplus} and~\eqref{eq:maxwellSplittingSchemeWithFilterpTildeOsci} to eliminate $\BBnp, \EEnplus$ and  $\EEnpminus$ from~\eqref{eq:maxwellSplittingSchemeWithFilterpTildeEnp}
\begin{align*}
	\EEnp
	& =
	-\OOmega \sin(\tau \OOmega) \ppn + \cos(\tau \OOmega) \EEn
	+ \tau \tfrac 1 2 (\cos(\tau \OOmega) + \Id) \psi_E(\tfrac \tau 2 \OOmega) \CCB \phi_B(\tfrac \tau 2 \OOmega) \BBn \\
	& \qquad \qquad +
	\tau^2 \tfrac 1 4 (\cos(\tau \OOmega) + \Id) \psi_E(\tfrac \tau 2 \OOmega) \CCB \phi_B(\tfrac \tau 2 \OOmega) \psi_B(\tfrac \tau 2 \OOmega) \CCE \phi_E(\tfrac \tau 2 \OOmega) \EEn.
\end{align*}
The filter functions $\psi_i,\phi_i$, $i \in \{E,B\}$ are even and 
hence the matrix-functions that are applied to $\ppn$ and $\BBn$ are uneven as functions of $\tau$, 
whereas the matrix-functions that are applied to $\EEn$ are even in $\tau$. This observation results
in the two step formulation
\begin{align*}
	&\EEnp - 2 \cos(\tau \OOmega) \EEn + \EE_{n-1} \\
	&\qquad \qquad =
	\tau^2 \tfrac 1 2 (\cos(\tau \OOmega) + \Id) \psi_E(\tfrac \tau 2 \OOmega) \CCB \phi_B(\tfrac \tau 2 \OOmega) \psi_B(\tfrac \tau 2 \OOmega) \CCE \phi_E(\tfrac \tau 2 \OOmega) \EEn
\end{align*}
To obtain a formulation close to the two step form of~\cite[Chapter XIII]{HaiLW06} we 
use~\eqref{eq:filterAssumptionsBequiv1} and get rid of the filter functions 
``between'' the two curl operators 
\begin{align}
  \EEnp - 2 \cos(\tau \OOmega) \EEn + \EE_{n-1}
  \label{eq:maxwellSplittingSchemeMultistepFormulation}
  &  =
  \tau^2 \tfrac 1 2 (\cos(\tau \OOmega) + \Id) \psi_E(\tfrac \tau 2 \OOmega) \GG \phi_E(\tfrac \tau 2 \OOmega) \EEn.
\end{align}
Again with~\eqref{eq:filterAssumptionsBequiv1} the equations for \BB and \pp of the numerical scheme finally read
\begin{align}
  \notag
  \ppnp & = \cos(\tau \OOmega) \ppn + \tau \sinc(\tau \OOmega) \EEn
  + \tau^2 \tfrac 1 2 \sinc(\tau \OOmega) \psi_E(\tfrac \tau 2 \OOmega) \CCB \BBn \\
  \label{eq:maxwellSplittingSchemeWithFilterFullStep_p}
  & \qquad - \tau^3 \tfrac 1 4 \sinc(\tau \OOmega) \psi_E(\tfrac \tau 2 \OOmega) \GG
  \phi_E(\tfrac \tau 2 \OOmega)   \EEn, \\
  \intertext{and}
  \label{eq:maxwellSplittingOneStep_B}
  \BBnp & = \BBn - \tau \tfrac 1 2 \CCE \phi_E(\tfrac \tau 2 \OOmega) \left( \EEn + \EEnp \right)\!.
\end{align}

\subsection{Error in the electric field}
\label{sec:estimatingE}
We want to apply Theorem 4.1 \cite[Chapter XIII]{HaiLW06} to estimate the
error in the electric field. Unfortunately this requires a distinct first time
step, that our scheme \eqref{eq:timeIntegrationScheme} does not
fulfill. To circumvent this problem we perturb the initial
value for the derivative of the \EE-field, which then yields the correct
scheme. For an estimate with the original initial values we
use a stability estimate for the exact solution.

The following theorem restates~\cite[Theorem XIII.4.1]{HaiLW06} adapted to the situation at hand.
\begin{theorem}
  \label{thm:hlwAdoption}
  Let $\OOmega$ and $\GG$ be as in Assumption~\ref{ass:physicalAndDiscretizationAssumptions}.
  Consider the solution of equation~\eqref{eq:modFourierModelEqMaxwellODE} for the electric field 
  by method~\eqref{eq:maxwellSplittingSchemeMultistepFormulation} with step size $\tau \leq \tau_0$ 
  for a sufficiently small $\tau_0$ independent of $\~\omega$ with $\tau \~\omega \geq c_0 > 0$.
  We denote the exact solution by $\EE\,\!'(t)$, and the numerical solution by $\EE_n\,\!'$.
  The first time step is computed via
  \begin{equation}
    \label{eq:distinctFirstTimeStep}
    \EE_1\,\!' = \cos(\tau \OOmega) \EE_0\,\!' + \tau \sinc(\tau \OOmega) \dot{\EE}_0\,\!' 
    + \tau^2 \tfrac 1 4 (\cos(\tau \OOmega) + \Id) \psi_E(\tfrac \tau 2 \OOmega) \GG \phi_E(\tfrac \tau 2 \OOmega) \EE_0\,\!',
  \end{equation}
  where 
  \begin{equation}
    \label{eq:modifiedInitialValues}
    \EE_0\,\!' := \EE_0,
    \ \text{ and }\
    \dot{\EE}_0\,\!' := \chi(\tau \OOmega) \CCB \BB_0 - \OOmega^2 \pp_0
  \end{equation}
  with
  \begin{equation}
    \label{eq:modifiedInitialValueFirstStepChi}
    \chi(z)
    := \tfrac 1 2 \frac{\cos(z) + 1}{\sinc(z)} \psi_E(\tfrac 1 2 z)
  \end{equation}
  and $\EE_0$ and $\BB_0$ and $\pp_0$ are such that conditions~\eqref{eq:initialValueConditions} hold true.
  
  If conditions \eqref{eq:filterAssumptionPsi}, \eqref{eq:filterAssumptionPhi}, \eqref{eq:filterAssumptionPsiTimesPhi} 
  and \eqref{eq:filterAssumptionPsi_Sinc} are satisfied with constants independent of $\~\omega$ 
  for the even entire filter functions 
  $\psi_E, \phi_E : \RR_{\geq 0} \to \RR$ with $\psi_E(0) = \phi_E(0) = 1$
  we obtain
  \begin{equation*}
    \| \EE_n\,\!' - \EE\,\!'(t_n) \| \leq C \tau^2
    \qquad
    \textup{for}
    \qquad
    t_n := t_0 + n \tau \leq T
  \end{equation*}
  with a constant $C$ independent of $n$, $\tau$ and $\tilde \omega$ but depending on the $(T - t_0)$ and the constants $H_0$, $C_g$ and $C_1, ..., C_3$.
\end{theorem}

\begin{proof}
  The filter functions of \cite[Theorem XIII.4.1]{HaiLW06} are
  \begin{equation}
    \label{eq:filterFunctionsHLW}
    \psi(z) := \tfrac 1 2 (\cos(z) + 1) \psi_E(\tfrac 1 2 z), \qquad \phi(z) := \phi_E(\tfrac 1 2 z), \forall ~ z \in \CC.
  \end{equation} 
  As $\GG$ is symmetric we can write $\GG \EE = \nabla
  U(\EE)$ with $U = \tfrac 1 2 \EE^T \GG \EE$.  From
  condition~\eqref{eq:filterAssumptionPsi_Sinc} we obtaion that $\chi$
  is bounded by $C_4$.  The factor $4/C_4^2$ in the estimate for $\CCB
  \BB_0$ in \eqref{eq:initialValueConditions} guaranties the estimate
  for the initial oscillatory energy for the perturbed initial values.
  Hence all assumptions of \cite[Theorem XIII.4.1]{HaiLW06} are
  fulfilled and its application completes the proof. \qed
\end{proof}

\begin{remark}
  \label{rem.2}
  In \cite[Theorem XIII.4.1]{HaiLW06} it is claimed that one would
  obtain
  $$
  \| \EE_n\,\!' - \EE\,\!'(t_n) \| \leq C \tau,
  $$
  if only the weaker estimate $|\psi(z)| \le C_0 |\sinc(\tfrac 1 2
  z)|$ instead of $|\psi(z)| \le C_1 \sinc(\tfrac 1 2 z)^2$
  holds. Numerical experiments with a linear $g$ i.e. $g(\EE) = \GG \EE$
  as the one considered in Section~\ref{sec:hlw} are a
  counterexample to this claim. It is this weaker estimate that is 
  fulfilled by the filter choice~\eqref{eq:filterChoiceJalo}.
\end{remark}

\begin{remark}
  \label{rem:correctScheme}
  By perturbing $\dot{\EE}_0$ of \eqref{eq:modFourierModelEqMaxwellIV} to $\dot{\EE}_0\,\!'$ of \eqref{eq:modifiedInitialValues} we have
  \begin{equation*}
    \EE_n = \EE_n\,\!',
    \qquad
    \forall ~ n \geq 0,
  \end{equation*}
  when computing $\EE_n$ with the original scheme \eqref{eq:timeIntegrationScheme} and $\EE_n\,\!'$ with the method described in 
  Theorem~\ref{thm:hlwAdoption}, so we can replace the numerical solution there by scheme \eqref{eq:timeIntegrationScheme}.
\end{remark}

To control the perturbation we apply the stability estimate of Lemma~\ref{lem:stabilityExactSolutionE} to the exact solution and obtain $\| \EE(t) - \EE\,\!'(t) \| \leq C \tau^2$, again with $C$ independent of $n$, $\tau$ and $\~\omega$.

\begin{lemma}
  \label{lem:stabilityExactSolutionE}
  Consider the exact solution $\Delta \EE(t)$ of \eqref{eq:modFourierModelEqMaxwellODE} with the initial values
  $\Delta\EE(t_0) = 0$ and $\partial_t \Delta \EE(t_0) = (\Id - \chi(\tau \OOmega)) \CCB \BB_0$ for a given $\BB_0$ and
  $\chi$ from \eqref{eq:modifiedInitialValueFirstStepChi} such that 
  Assumption~\ref{ass:physicalAndDiscretizationAssumptions} holds.
  
  If the filter functions $\psi_E, \phi_E : \RR_{\geq 0} \to \RR$ are entire functions with $\psi_E(0) = \phi_E(0) = 1$
  satisfying \eqref{eq:filterAssumptionsBequiv1} and \eqref{eq:filterAssumptionSincMinusPsi} 
  with constant $C_5$ independent of $\~\omega$,
  then
  \begin{equation*}
    \| \Delta \EE(t) \| \leq C \tau^2
  \end{equation*}
  with a constant $C$ independent of $n$, $\tau$ and $\~\omega$, but depending on $(T - t_0)$, $C_5$ and $H_0$.
\end{lemma}

\begin{proof}
  By Assumption~\ref{ass:physicalAndDiscretizationAssumptions} the matrices $-\GG$ and $\OOmega^2$ are both symmetric  positive semidefinite, so are their sum and we can define the symmetric positive semidefinite matrix $\mathbf{B} := \sqrt{\OOmega^2 - \GG}$.
  Using the matrix sinc function the exact solution is
  \begin{align*}
    \Delta\EE(t)
    & = \cos((t - t_0) \mathbf{B}) \Delta\EE(t_0) + (t - t_0) \sinc((t - t_0) \mathbf{B}) \partial_t \Delta \EE(t_0) \\
    & = (t - t_0) \sinc((t - t_0) \mathbf{B}) (\Id - \chi(\tau \OOmega)) \CCB \BB_0.
  \end{align*}
  Since $| z \sinc(z)| \leq 1$, $z \geq 0$, we only have to control the real part.
  Condition~\eqref{eq:filterAssumptionSincMinusPsi} yields
  \begin{equation*}
    |1 - \chi(z)| \leq C_5 z^2
    \quad
    \Rightarrow
    \quad
    \| (\Id - \chi(\tau \OOmega)) \CCB \BB_0 \|
    \leq
    C_5 \tau^2 \| \OOmega^2 \CCB \BB_0 \|.
  \end{equation*}
  This gives the desired bound for $C = (T - t_0) C_5 \sqrt{H_0}$.\qed
\end{proof}

Combining Theorem \ref{thm:hlwAdoption} and Lemma \ref{lem:stabilityExactSolutionE} the error estimate for the electric field is obtained:
\begin{theorem}
  \label{thm:errorE}
  Let $\GG$, $\OOmega$ and the initial values $\pp_0$, $\EE_0$ and $\BB_0$ be as in 
  Assumptions~\ref{ass:physicalAndDiscretizationAssumptions} and \ref{ass:initialData}.
  Consider the numerical solution of \eqref{eq:maxwellsEquationsDiscrete} with the 
  scheme~\eqref{eq:timeIntegrationScheme} with step size $\tau$ satisfying $\tau \leq \tau_0$ 
  for sufficiently small $\tau_0$ independent of $\~\omega$ and $\tau \~\omega \geq c_0 > 0$.
  Denote the exact solution for the electric field by $\EE(t)$ and the numerical one by $\EE_n$.
  
  If the filter functions satisfy \eqref{eq:filterAssumptionsBequiv1} 
  and the assumptions \eqref{eq:filterAssumptionPsi} to \eqref{eq:filterAssumptionSincMinusPsi}
  for constants independent of $\~\omega$, the
  \begin{equation*}
    \| \EE_n - \EE(t_n) \| \leq C \tau^2
    \qquad
    \textup{for}
    \qquad
    t_n := t_0 + n \tau \leq T
  \end{equation*}
  with a constant $C$ independent of $n$, $\tau$ and $\omega$ but depending on the $(T - t_0)$ 
  and the constants $H_0$, $C_g$ and $C_1, ..., C_5$.
  \end{theorem}

\begin{proof}
  By Remark \ref{rem:correctScheme} scheme \eqref{eq:timeIntegrationScheme} with adjusted initial value \eqref{eq:modifiedInitialValues} with \eqref{eq:modifiedInitialValueFirstStepChi} is the same as the two step scheme with the first step \eqref{eq:distinctFirstTimeStep} from Theorem \ref{thm:hlwAdoption}.
  We again call the perturbed exact solution $\EE\,\!'(t)$.
  Since $\EE(t_0) - \EE\,\!'(t_0) = 0 = \Delta \EE(t_0)$ and $\partial_t \EE(t_0) - \partial_t \EE\,\!'(t_0) = \partial_t \Delta \EE(t_0)$ 
  from Lemma \ref{lem:stabilityExactSolutionE} we have $\| \Delta \EE(t_n) \| \leq C \tau^2$.
  From this and with Theorem \ref{thm:hlwAdoption} we then obtain
  \begin{equation*}
    \| \EE_n - \EE(t_n) \| \leq \| \EE\,\!'_n - \EE\,\!'(t_n) \| + \| \Delta \EE(t_n) \| \leq C \tau^2.
  \end{equation*}
  The constant $C$ has the stated dependencies. \qed
\end{proof}

\subsection{Error in the magnetic flux}
\label{sec:estimatingB}

To estimate the error in the magnetic flux $\BB$, filter assumption~\eqref{eq:filterAssumptionSincMinusPhi} is needed.

\begin{theorem}
  \label{thm:errorB}
  Suppose the assumptions  of Theorem~\ref{thm:errorE} hold 
  and additionally assume~\eqref{eq:filterAssumptionSincMinusPhi} holds with $C_6$ independent of $\~\omega$.
  Then for $t_n := t_0 + n \tau \leq T$ 
  \begin{equation*}
    \| \BB_n - \BB(t_n) \| \leq C \tau^2
  \end{equation*}
  with a constant $C$ independent of $n$, $\tau$ and $\omega$ but depending on the $(T - t_0)$ 
  and the constants $H_0$, $C_g$ and $C_1$ to $C_6$.
  
\end{theorem}

\begin{proof}
  From equations \eqref{eq:maxwellsEquationsDiscrete_B} and \eqref{eq:maxwellSplittingOneStep_B} we obtain 
  the recursion of the error in \BB
  \begin{equation*}
    \BB(t_{n+1}) - \BBnp
    = \BB(t_n) - \BBn - \frac{\tau}{2} \CCE \left(
      2 \int_0^1 \EE(t_n + \tau s) \dd\!s
      - \phi_E(\tfrac \tau 2 \OOmega) (\EEn + \EEnp)
    \right).
  \end{equation*}
  Applying the variation of constants formula \eqref{eq:EfieldFormula} to \eqref{eq:modFourierModelEqMaxwellODE} for
  the argument $t_n + \tau s = t_{n+1} - (1 -\tau) s$ to expand around $t_n$ and at the same time around $t_{n+1}$ we obtain
  for the term in parentheses 
  \begin{align*}
    &
    2 \int_0^1 \EE(t_n + \tau s) \dd\!s
    - \phi_E(\tfrac \tau 2 \OOmega) (\EEn + \EEnp)
    \\
    & =
    \left( \int_0^1 \cos(\tau s \OOmega) \dd\!s - \phi_E(\tfrac \tau 2 \OOmega) \right) \EE(t_n)
    + \left( \int_0^1 \cos(\tau (s - 1) \OOmega) \dd\!s - \phi_E(\tfrac \tau 2 \OOmega) \right) \EE(t_{n+1}) \\
    & \qquad \qquad +
    \tau \left( \int_0^1 s \sinc(\tau s \OOmega) \dd\!s ~ \partial_t \EE(t_n)
      + \int_0^1 (s - 1) \sinc(\tau (s - 1) \OOmega) \dd\!s ~ \partial_t \EE(t_{n+1}) \right) \\
    & \qquad \qquad +
    \tau^2 \left( \int_0^1 s^2 I_n^+(\tau, s) \dd\!s + \int_0^1 (1 - s)^2 I_{n+1}^-(\tau, s) \dd\!s\right) \\
    & \qquad \qquad +
    \phi_E(\tfrac \tau 2 \OOmega) \left( (\EE(t_n) - \EEn) + (\EE(t_{n+1}) - \EEnp) \right),
  \end{align*}
  where $I^+_n$ and $I^-_{n+1}$ are bounded independently of $\~\omega$ 
  containing the convolution terms of the variation of constants formula.
  Here we use the boundedness of $\sinc$ and the bounds on $\EE(t)$ from \eqref{eq:EFieldBounded}.
  
  Computing the integrals and adding up the errors of all time steps yields
  \begin{subequations}
    \label{eq:errorB}
    \begin{align}
      &\BB(t_n) - \BBn = \notag
      \\
      \label{eq:errorB_E}
      & - \frac \tau 2 \CCE \sum_{l=0}^{n-1} \left[
        \left( \sinc(\tau \OOmega) - \phi_E(\tfrac \tau 2 \OOmega) \right) \EE(t_l)
        + \left( \sinc(\tau \OOmega) - \phi_E(\tfrac \tau 2 \OOmega) \right) \EE(t_{l+1})
      \right] \\
      \label{eq:errorB_dotE}
      & - \frac {\tau^2} 2 \CCE \sum_{l=0}^{n-1} \left[
        \cosc(\tau \OOmega) \partial_t \EE(t_l) - \cosc(\tau \OOmega) \partial_t \EE(t_{l+1})
      \right] \\
      \label{eq:errorB_Integral}
      & - \frac {\tau^3} 2 \CCE \sum_{l=0}^{n-1} \left[
        \int_0^1 s^2 I_l^+(\tau, s) \dd\!s + \int_0^1 (1 - s)^2 I_{l+1}^-(\tau, s) \dd\!s
      \right] \\
      \label{eq:errorB_ErrorE}
      & - \frac \tau 2 \CCE \sum_{l=0}^{n-1} \phi_E(\tfrac \tau 2 \OOmega) \left[
        (\EE(t_l) - \EE_l) + (\EE(t_{l+1}) - \EE_{l+1})
      \right].
    \end{align}
  \end{subequations}
  with the even entire function $\cosc := z \mapsto \int_0^1 \cos((1 - \xi) z) \xi \dd\!\xi$ 
  satisfying $z^2 \cosc(z) = 1 - \cos(z)$ and $|\cosc(z)|\leq \tfrac 1 2$, $z \in \RR$.
  
  We use the bound \eqref{eq:filterAssumptionPhi} on $\phi_E$, the
  bound on $\CCE$ and the $\mathcal{O}(\tau^2)$ estimate of
  Theorem~\ref{thm:errorE} to bound \eqref{eq:errorB_ErrorE} by $C
  \tau^2$, where we loose one factor $\tau$ due to summing up.  The
  bound for \eqref{eq:errorB_Integral} follows from the boundedness of
  $I^{\pm}_l$.  \eqref{eq:errorB_dotE} is
  a telescopic sum, so we do not loose a $\tau$ by summing up. The
  boundedness of the Hamiltonian in \eqref{eq:HamiltonBounded} yields a
  bound for the $\partial_t \EE(t)$. The boundedness of $\cosc$ then
  yields the second order estimate for \eqref{eq:errorB_dotE}.
  
  To control \eqref{eq:errorB_E} we apply the variation of constants formula \eqref{eq:EfieldFormula}, 
  with $t'=t_0$ to obtain $\EE(t_l)$.
  \begin{subequations}
    \label{eq:errorBvoc}
    \begin{align}
      \notag
      &  \eqref{eq:errorB_E}  = 
       \frac \tau 2 \CCE \sum_{l=0}^{n-1}
        \left[ \sinc(\tau \OOmega) - \phi_E(\tfrac \tau 2 \OOmega) \right] (\EE(t_l) + \EE(t_{l+1}))
       \\
      \label{eq:errorBvoc_cos}
      & = 
       \frac \tau 2 \CCE \sum_{l=0}^{n-1} \left[ \sinc(\tau \OOmega) - \phi_E(\tfrac \tau 2 \OOmega) \right] \left[ \cos(l \tau \OOmega) + \cos((l + 1) \tau \OOmega) \right] \EE_0 \\
      \label{eq:errorBvoc_sinc}
      & +
       \frac \tau 2 \CCE \sum_{l=0}^{n-1} \left[ \sinc(\tau \OOmega) - \phi_E(\tfrac \tau 2 \OOmega) \right] \left[ l \tau \sinc(l \tau \OOmega) + (l + 1) \tau \sinc((l + 1) \tau \OOmega) \right] \dot\EE_0 \\
      \label{eq:errorBvoc_integral}
      & +
       \tau \CCE \sideset{}{'}\sum_{l=0}^n \left[ \sinc(\tau \OOmega) - \phi_E(\tfrac \tau 2 \OOmega) \right] \left[ l^2 \tau^2 \int_0^1 (1 - \xi) \sinc(l \tau (1 - \xi) \OOmega) \GG \EE(t_0 + \tau l \xi)
        \right],
    \end{align}
  \end{subequations}
  where the prime in the summation indicates that the first and last term are weighted by $\tfrac 1 2$.
  At first sight the norm of each of the three terms (\ref{eq:errorBvoc}a,b,c) seems to 
  be in $\mathcal{O}(1)$. 
  
  To show that they are actually in $\mathcal{O}(\tau^2)$ we use the identities
  \begin{equation*}
    \cos(l z)
    = \frac{\sin((l + \tfrac 1 2) z) - \sin((l - \tfrac 1 2) z)}{2 \sin(\tfrac 1 2 z)},
    \,
    \,
    l \sinc(l z)
    = -\frac{\cos((l + \tfrac 1 2) z) - \cos((l - \tfrac 1 2) z)}{2 z \sin(\tfrac 1 2 z)}.
  \end{equation*}
  These allow to simplify the sum of cosines and sincs in~\eqref{eq:errorBvoc_cos} and~\eqref{eq:errorBvoc_sinc}
  respectively.
  \begin{equation}
    \label{eq:errorB_TrigonometricSummation}
    \begin{aligned}
      \left( \sinc(z) - \phi_E(\tfrac 1 2 z) \right)
      &\left( \sum_{l=0}^{n-1} \cos(l z) + \sum_{l=1}^{n} \cos(l z) \right) \\
      & = \sin(n z) \cos(\tfrac 1 2 z) z \frac{\sinc(z) - \phi_E(\tfrac 1 2 z)}{z\sin(\tfrac 1 2 z)}, \\
      \left( \sinc(z) - \phi_E(\tfrac 1 2 z) \right)
      &\left( \sum_{l=0}^{n-1} l \sinc(l z) + \sum_{l=1}^{n} l \sinc(l z) \right) \\
      & = -(\cos(n z) - 1) \cos(\tfrac 1 2 z) \frac{\sinc(z) - \phi_E(\tfrac 1 2 z)}{z \sin(\tfrac 1 2 z)}
    \end{aligned}
  \end{equation}
  The trigonometric functions multiplying the fractions on the right hand sides above are bounded, 
  such that it suffices to control
  \begin{equation*}
    \chi_0(z) := \frac{\sinc(z) - \phi_E(\tfrac 1 2 z)}{z \sin(\tfrac 1 2 z)}.
  \end{equation*}
  This is the place where we finally use the new filter assumption~\eqref{eq:filterAssumptionSincMinusPhi} 
  to obtain
  \begin{equation}
    | \chi_0(z) | \leq C_6,
  \end{equation}
  such that potential new singularities are controlled. We obtain
  \begin{equation}
    \|\eqref{eq:errorBvoc_cos}\|
    \leq \tfrac \tau 2 C_c \| \sin(n \tau \OOmega) \cos(\tfrac 1 2 \tau \OOmega) \| \| \chi_0(\tau \OOmega) \tau 
    \OOmega \EE_0 \|
    \leq \tfrac \tau 2 C_c C_6 \tau \sqrt{\tfrac 2 3 H_0},
  \end{equation}
  since $\sin$ and $\cos$ are bounded by one, $\chi_0$ by $C_6$ and $||\OOmega\EE_0||$ by $\tfrac 2 3 H_0$,   
  cf.~\eqref{eq:initialValueConditions}. Likewise we have
  \begin{equation}
    \|\eqref{eq:errorBvoc_sinc}\|
    \leq \tfrac 1 2 \tau^2 C_c \| (\cos(n \tau \OOmega) - \Id) \cos(\tfrac 1 2 \tau \OOmega) \| \| \chi_0(\tau \OOmega) \dot\EE_0 \|
    \leq \tfrac 1 2 \tau^2 C_c 2 C_6 \sqrt{2 H_0},
  \end{equation}
  where $\dot\EE_0$ is bounded by the Hamiltonian in \eqref{eq:HamiltonBounded}.
  
  This way we used the filter function $\phi_E$ to filter periodic singularities.
  This is the reason why we need $\sinc$ terms on the right hand side of the filter assumptions.
  For the remainder we use it to filter out higher order singularities in a neighborhood of zero, 
  that leads to factors of $z$ on the right hand side in the filter assumptions.
  
  It remains to bound the integral term of the summand \eqref{eq:errorBvoc_integral}, 
  that is we need an $\mathcal{O}(1)$ bound on
  \begin{equation}
    \label{eq:errorB_E_IntegralShortcut}
    J_l := \tau \mapsto l^2 \vartheta_0(\tau \OOmega) \int_0^1 (1 - \xi) \sinc(l \tau (1 - \xi) \OOmega) 
    f(\xi) \dd\!\xi
  \end{equation}
  for
  \begin{equation}
    \label{eq:errorB_E_IntegralShortcut_f}
    f := \xi \mapsto \GG \EE(t_0 + \tau l \xi),
  \end{equation}
  and the auxillary functions
  \begin{equation}
    \label{eq:errorB_E_Integral_theta}
    \vartheta_i(z) := \frac{\sinc(z) - \phi_E(\tfrac 1 2 z)}{z^i},
    \qquad
    i \in \lbrace 0, 1, 2 \rbrace.
  \end{equation}
  These functions $\vartheta$ satisfy the relations
  \begin{equation}
    \label{eq:errorB_E_Integral_thetaRelations}
    z \vartheta_{i}(z) = \vartheta_{i-1}(z),
    \qquad
    i \in \lbrace 1, 2 \rbrace.
  \end{equation}
  where the first one in turn yields
  \begin{equation*}
    l \vartheta_0(z) (1 - \xi) \sinc(l (1 - \xi) z) = \vartheta_1(z) \sin(l (1 - \xi) z).
  \end{equation*}
  The filter assumption \eqref{eq:filterAssumptionSincMinusPhi} applied directly gives an $\mathcal{O}(1)$ bound on $\vartheta_1$, which in turn leads to a $\mathcal{O}(n)$ bound on $J_l$ for $l=1,\dots,n$ 
  and thus to a first order estimate for the magnetic flux.
  
  To improve this estimate we use the identity $z \sinc(z) = \sin(z)$ which gives an 
  even sharper estimate on the filtering abilities of $\phi_E$ by
  \begin{equation}
    | \sinc(z) - \phi_E(\tfrac 1 2 z) |
    \leq C_6 | \tfrac 1 2 z^2 \sinc(\tfrac 1 2 z) |
    \leq \tfrac 1 2 C_6 z^2
  \end{equation}
  and thus an $\mathcal{O}(1)$ bound on $\vartheta_2$, since the $\sinc$ function is bounded by one.
  
  To make use of this estimate we use
  \begin{equation*}
    \vartheta_1(z) l \sin(l (1 - \xi) z)
    \stackrel{\eqref{eq:errorB_E_Integral_thetaRelations}}{=}
    \vartheta_2(z) l z \sin(l (1 - \xi) z)
    = \frac{\partial}{\partial \xi} \vartheta_2(z) \cos(l (1 - \xi) z).
  \end{equation*}
  Integration by parts of $J_l$ yields
  \begin{align*}
    \int_0^1 \vartheta_1(z) l \sin(l (1 - \xi) z) f(t_0 + l \tau \xi) \dd\!\xi
    =
    & \bigl[ \vartheta_2(z) \cos(l (1 - \xi) z) f(t_0 + l \tau \xi) \bigr|_{\xi = 0}^{\xi=1} \\
    & - \int_0^1 \vartheta_2(z) \cos(l (1 - \xi) z) \frac{\partial f}{\partial \xi} (t_0 + l \tau \xi) l \tau \dd\!\xi.
  \end{align*}
  Since by definition of $f$ in~\eqref{eq:errorB_E_IntegralShortcut_f} we have
  \begin{equation*}
    \tfrac{\dd}{\dd\!\xi} f = \xi \mapsto l \tau \GG \partial_t \EE(t_0 + \tau l \xi)
  \end{equation*}
  this applies to $J_l$ by
  \begin{align*}
    J_l(\tau)
    & =
    \vartheta_2(\tau \OOmega) \GG \EE(t_0 + l \tau)
    - \vartheta_2(\tau \OOmega) \cos(l \tau \OOmega) \GG \EE_0
    \\
    & \qquad \qquad
    - \int_0^1 \vartheta_2(\tau \OOmega) \cos(l \tau (1 - \xi) \OOmega) l \tau \GG \partial_t \EE(t_0 + l \tau \xi) \dd\!\xi.
  \end{align*}
  The boundedness of $\GG$ in Assumption~\ref{ass:physicalAndDiscretizationAssumptions} 
  and the stability estimates in \eqref{eq:HamiltonBounded} and \eqref{eq:EFieldBounded} 
  allow to control $\GG \EE$ and $\GG \partial_t \EE$. 
  All the matrix functions are bounded and $l \tau \leq T - t_0$ 
  such that $J_l(\tau)$ is in $\mathcal{O}(1)$ with a constant 
  independent of $\~\omega$.
  
  This concludes the proof for the error in the magnetic flux.
  \qed
\end{proof}

\begin{remark}
  By choosing $\phi_E(z) = \sinc(2 z)$ the left hand side of~\eqref{eq:filterAssumptionSincMinusPhi} 
  and the  term~\eqref{eq:errorB_E}  vanish, thus one may set $C_6 = 0$ and the above proof is simplified drastically. 
  It is shown in~\cite{Jan15} that the choice $\phi_E(z) = \sinc(2 z)$ is indeed a valid choice
  and respects the filter conditions~\eqref{eq:filterAssumptionPhi},~\eqref{eq:filterAssumptionPsiTimesPhi},~\eqref{eq:filterAssumptionSincMinusPhi} and~\eqref{eq:filterAssumptionSinc2MinusSincPhi}.
  
 \end{remark}

\subsection{Error in the impulses}
\label{sec:estimatingp}

To conclude the proof of the main result, Theorem~\ref{thm:errorEstimateFinal}, we have to show the corresponding estimate for the error in the impulses. This is where the last two assumptions on the filter functions \eqref{eq:filterAssumptionPsiBounded} and \eqref{eq:filterAssumptionSinc2MinusSincPhi} enter.

\begin{theorem}
  \label{thm:errorp}
  Suppose the assumptions of Theorem \ref{thm:errorB} hold. If
  \eqref{eq:filterAssumptionPsiBounded} and \eqref{eq:filterAssumptionSinc2MinusSincPhi} hold with $C_7$ and $C_8$
  independent of $\~\omega$ then
  \begin{equation*}
    \| \pp_n - \pp(t_n) \| \leq C \tau^2
    \qquad
    \textup{for}
    \qquad
    t_n := t_0 + n \tau \leq T,
  \end{equation*}
  with a constant $C$ independent of $n$, $\tau$ and $\omega$ 
  but depending on the $(T - t_0)$ and the constants $H_0$, $C_g$ and $C_1$ to $C_8$.
%
\end{theorem}

\begin{proof}
%
  We start by expressing the impulses with the fundamental theorem of calculus applied to the differential equation for the impulses \eqref{eq:maxwellsEquationsDiscrete_p}. Applying \eqref{eq:EfieldFormula} gives a formula for the
  exact solution of the electric field.
  The numerical solution is expressed by~\eqref{eq:maxwellSplittingSchemeWithFilterFullStep_p}. Then the error in the $(n+1)$st step reads
  \begin{subequations}
    \label{eq:errorp}
    \begin{align}
      \notag
      \pp(t_{n+1}) - \ppnp
      & = \pp(t_n)
      + \tau \left( \int_0^1 \cos(\tau s \OOmega) \dd\!s \EE(t_n) - \sinc(\tau \OOmega) \EEn \right) \\
      \notag
      & \qquad
      + \tau \int_0^1 \tau s \sinc(\tau s \OOmega) \dd\!s \;\! \partial_t\EE(t_n)
      - \cos(\tau \OOmega) \ppn \\
      \notag
      & \qquad
      - \tfrac {\tau^2} 2 \sinc(\tau \OOmega) \psi_E(\tfrac \tau 2 \OOmega) \CCB \BBn
      + \tau^3 I_n(\tau) \\
      \notag
      & =
      \pp(t_n) + \tau \sinc(\tau \OOmega) (\EE(t_n) - \EEn) + \tau^3 I_n(\tau) \\
      \label{eq:errorp_B}
      & \qquad
      + \tau^2 \left( \cosc(\tau \OOmega) \CCB \BB(t_n) - \tfrac 1 2 \sinc(\tau \OOmega) \psi_E(\tfrac \tau 2 \OOmega) \CCB \BBn \right) \\
      \label{eq:errorp_p}
      & \qquad
      -\tau^2 \OOmega^2 \cosc(\tau \OOmega) \pp(t_n) - \cos(\tau \OOmega) \pp_n,
    \end{align}
  \end{subequations}
  where
  \begin{align}
    \notag
		&I_n(\tau) : = \int_0^1 s^2 \int_0^1 (1 - \xi) \sinc(\tau s (1 - \xi) \OOmega) \GG \EE(t_n + \tau s \xi) \dd\!\xi \dd\!s \\
		\label{eq:errorp_In}
		&\hspace{5cm} - \tfrac 1 4 \sinc(\tau \OOmega) \psi_E(\tfrac \tau 2 \OOmega) \GG \phi_E(\tfrac \tau 2 \OOmega) \EEn.
	\end{align}
	The $\cosc$ function was already used in \eqref{eq:errorB_dotE} for the estimate for $\BB$ and can also be written as an 
        integral over $\xi \sinc(\xi z)$.
	The filter estimate \eqref{eq:filterAssumptionPsiBounded} yields the boundedness of $\psi_E$,  
        with $\| \EE_n \| \leq \| \EE(t_n) \| + \| \EE_n - \EE(t_n) \|$, the estimate for $\EE$ of Theorem \ref{thm:errorE} 
        and the stability estimate for the electric field~\eqref{eq:EFieldBounded} we obtain
	\begin{equation}
		\label{eq:errorp_normIn}
		\| I_n(\tau) \| \leq C_I,
	\end{equation}
	with a constant $C_I$ independent of $\~\omega$, since $\tau \leq \tau_0$.
	For \eqref{eq:errorp_p} we use $z^2 \cosc(z) = 1 - \cos(z)$ to retrieve
	\begin{equation}
		\label{eq:errorp_pExpression}
		-\tau^2 \OOmega^2 \cosc(\tau \OOmega) \pp(t_n) - \cos(\tau \OOmega) \ppn
		= \cos(\tau \OOmega) (\pp(t_n) - \ppn) - \pp(t_n),
	\end{equation}
	for \eqref{eq:errorp_B} anlogously with $\cosc(2 z) = \frac 1 2 \sinc^2(z)$
	\begin{align}
		\notag
		& \tau^2 \left( \cosc(\tau \OOmega) \CCB \BB(t_n) - \tfrac 1 2 \sinc(\tau \OOmega) \psi_E(\tfrac \tau 2 \OOmega) \CCB \BBn \right) \\
		\label{eq:errorp_BExpression}
		\notag
		& \qquad \qquad \qquad \qquad =
		\frac {\tau^2} 2 \left( (\sinc^2(\tfrac \tau 2 \OOmega) - \sinc(\tau \OOmega) \psi_E(\tfrac \tau 2 \OOmega)) \CCB \BB(t_n) \right. \\
		& \qquad \qquad \qquad \qquad \qquad \qquad
		+ \left. \sinc(\tau \OOmega) \psi_E(\tfrac \tau 2 \OOmega) \CCB (\BB(t_n) - \BB_n) \right).
	\end{align}
	We define the next auxiliary function
	\begin{equation*}
		J_n(\tau):= \tau \left(
			\sinc(\tau \OOmega) (\EE(t_n) - \EEn)
			+ \tau\sinc(\tau \OOmega) \psi_E(\tfrac \tau 2 \OOmega) \CCB (\BB(t_n) - \BB_n)
			+ \tau^2 I_n(\tau)
		\right).
	\end{equation*}
	This, the boundedness of $\sinc$, $\psi_E$ und $\CCB$ and the error estimates for $\EE$ and $\BB$ from Theorems \ref{thm:errorE} and \ref{thm:errorB} yields the second order estimate
	\begin{equation}
		\label{eq:errorp_normJn}
		\| J_n(\tau) \|
		\leq C_E \tau^2 + \tau C_7 C_c C_B \tau^2 + C_I \tau^2 =: C_J \tau^2
	\end{equation}
	for $J_n(\tau)$.
	Resolving the recursion in \eqref{eq:errorp} we get the summed error
	\begin{align}
		\notag
		&
		\pp(t_n) - \ppn
		= \tau \sum_{l=0}^n \cos^l(\tau \OOmega) J_{n-l-1}(\tau) \\
		\label{eq:errorp_resolved}
		& \qquad
		+ \tfrac{\tau^2}{2} \sum_{l=0}^n \cos^l(\tau \OOmega) (\sinc^2(\tfrac \tau 2 \OOmega) - \sinc(\tau \OOmega) \psi_E(\tfrac \tau 2 \OOmega)) \CCB \BB(t_{n-l-1}).
	\end{align}
	The first summand with $J_n(\tau)$ and the leading factor of $\tau$ is of right order due to \eqref{eq:errorp_normJn}.
	The second summand seems to be of too low order to succeed with a global error proof of second order.
	We have to use the trigonometric identity
	\begin{equation*}
		\cos^n(z) = \frac{\cos^{n+1}(z) - \cos^{n}(z)}{-2\sin^2(\tfrac 1 2 z)}
	\end{equation*}
	and the filtering abilities of $\psi_E$ to avoid summing up of errors.
	With the help of parital summation
	\begin{equation*}
		\sum_{l=0}^{n-1} (f_{l+1} - f_l) g_l = \sum_{l=0}^{n-1} f_l (g_{l-1} - g_l) + f_n g_{n-1} - f_0 g_{-1},
	\end{equation*}
	with $f_l := \frac{\cos^l(z)}{-2\sin^2(\tfrac 1 2 z)}$ and $g_l := \CCB \BB(t_{n-l-1})$ the trigonometric identity yields
	\begin{align}
		\notag
		& (\sinc^2(\tfrac 1 2 z) - \sinc(z) \psi_E(\tfrac 1 2 z)) \sum_{l=0}^{n-1} \cos^l(z) \CCB \BB(t_{n-l-1}) \\
		\notag
		& \qquad \qquad= 
		\frac{\sinc^2(\tfrac 1 2 z) - \sinc(z) \psi_E(\tfrac 1 2 z)}{-2\sin^2(\tfrac 1 2 z)} \\
		\label{eq:errorp_BExpression_partialSummation}
		& \qquad \qquad \qquad
		\cdot \left(
			\sum_{l=1}^{n} \cos^l(z) \CCB (\BB(t_{n-l}) - \BB(t_{n-l-1})) + \cos^n(z) \CCB \BB_0 - \CCB \BB(t_n)
		\right).
	\end{align}
	Filter assumption \eqref{eq:filterAssumptionSinc2MinusSincPhi} gives us the estimate
	\begin{equation*}
		\left| \frac{\sinc^2(\tfrac 1 2 z) - \sinc(z) \psi_E(\tfrac 1 2 z)}{-2\sin^2(\tfrac 1 2 z)} \right| \leq \tfrac 1 2 C_8
	\end{equation*}
	for the singularities the appeared in \eqref{eq:errorp_BExpression_partialSummation} and thus
	\begin{align*}
		\| \text{r.h.s. of } \eqref{eq:errorp_BExpression_partialSummation} \|
		&
		\leq \tfrac 1 2 C_8 \left(\sum_{l=0}^{n-1} \| \CCB (\BB(t_{n-l}) - \BB(t_{n-l-1})) \| + \~C \right)
	\end{align*}
	using the boundedness of the magnetic flux \eqref{eq:BFieldBounded} to estimate the boundary terms $\CCB\BB_0$ and $\CCB\BB(t_n)$ with a constant $\~C$ independent of $\~\omega$. Since the boundary terms appear only once, it is sufficient that they are in $\mathcal{O}(1)$.
	
	To generate the last factor of $\tau$ we once more need to apply the fundamental theorem of calculus, this time on the analytical solution of the magnetic flux and substitute the right hand side of the differential equation for $\BB$ \eqref{eq:maxwellsEquationsDiscrete_B} in the time derivate:
	\begin{align*}
		& \| \CCB (\BB(t_{n-l}) - \BB(t_{n-l-1})) \| \\
		& \hspace{1cm} =
		\left\| \CCB \left( \BB(t_{n-l-1}) - \tau \int_0^1 \CCE \EE(t_{n-l-1} + \tau \xi) \dd\!\xi - \BB(t_{n-l-1}) \right) \right\| 
		\leq \tau \widehat{C},
	\end{align*}
	with another constant $\widehat{C}$ independent of $\~\omega$, using the boundedness of $\EE(t)$.
	The $\tau^2$ factor in front of the second sum in the error formula \eqref{eq:errorp_resolved} is thus sufficient for the global second order estimate.
	\qed
\end{proof}

\section{Multiple high frequencies}
\label{sec:multifrequency}
Consider now the case of multiple frequencies, i.e. let's assume that $\OOmega$ is a positive semi-definite matrix and that $\omega$ is a bound for its largest eigenvalue.
Modifying the results and the proof of Grimm and Hochbruck~\cite{GriH06}
a proof for the second order error estimate for the triple splitting method 
was obtained by Buchholz and Hochbruck in~\cite{BucH15}.

The only ingredient that is required in our convergence 
proof is a replacement for Theorem~\ref{thm:hlwAdoption}. 
We can use \cite[Theorem 1]{GriH06} of Grimm and Hochbruck directly by writing 
their scheme as a two step formulation for the solution 
(getting rid of its derivative). Again we have to perturb the initial values to
adjust to the situation at hand.

We use the multistep form \eqref{eq:maxwellSplittingSchemeMultistepFormulation} with destinct first step \eqref{eq:distinctFirstTimeStep} for the perturbed initial values \eqref{eq:modifiedInitialValues}.
As already stated in Remark \ref{rem:correctScheme} this is equivalent to our triple splitting method \eqref{eq:timeIntegrationScheme} with $\psi_B \equiv \phi_B \equiv 1$.
The two step formulation with the destinct first step is equivalent to \cite[Scheme (3)]{GriH06} with filter functions $\phi$ and $\psi$ as in \eqref{eq:filterFunctionsHLW}, $\psi(z) = \sinc(z) \psi_1(z)$ and $\psi_0(z) = \cos(z) \psi(z)$.
For a second order error estimate for scheme~\eqref{eq:timeIntegrationScheme} with $\psi_B \equiv \phi_B \equiv 1$ we then require \eqref{eq:filterAssumptionPsi_Sinc}-\eqref{eq:filterAssumptionSinc2MinusSincPhi} as before, but replace the first three assumptions \eqref{eq:filterAssumptionPsi}-\eqref{eq:filterAssumptionPsiTimesPhi} by
\begin{subequations}
\label{eq:filterAssumptionEGH}
\begin{align}
	\label{eq:filterAssumptionGH1}
	| 1 - \phi_E(\tfrac 1 2 z) |
	& \leq C_9 | z | \\
	\label{eq:filterAssumptionGH2}
	| \sinc^2(\tfrac 1 2 z) - \tfrac 1 2 (\cos(z) + 1) \psi_E(\tfrac 1 2 z)) |
	& \leq C_{10} | \sin(\tfrac 1 2 z) | \\
	\label{eq:filterAssumptionGH3}
	| \sinc(z) - \phi_E(\tfrac 1 2 z) |
	& \leq C_{11} | z \sin(\tfrac 1 2 z) | \\
	\label{eq:filterAssumptionGH4}
	| \sinc^2(z) - \tfrac 1 2 (\cos(z) + 1) \psi_E(\tfrac 1 2 z) |
	& \leq C_{12} | \sin(z) \sin(\tfrac 1 2 z) | \\
	\label{eq:filterAssumptionGH5}
	| \sinc^2(z) - \tfrac 1 2 (\cos(z) + 1) \psi_E(\tfrac 1 2 z) \cos(z) |
	& \leq C_{13} | \sin(z) \sin(\tfrac 1 2 z) |
\end{align}
\end{subequations}
Assumptions \eqref{eq:filterAssumptionPsi_Sinc} and \eqref{eq:filterAssumptionPsiBounded} yield
\begin{equation*}
	| \eta(z) | \leq \max\{ 2 C_4, C_7 \},
\end{equation*}
for $\eta \in \{ \phi, \psi, \psi_0, \psi_1 \}$
which is~\cite[Condition (11)]{GriH06}.
The new assumption~\eqref{eq:filterAssumptionGH1} yields
\begin{equation*}
	| (\phi(z) - 1) | \leq C_9 | z |,
\end{equation*}
which is~\cite[Condition (12)]{GriH06}.
\eqref{eq:filterAssumptionGH2} yields
\begin{equation*}
	| (\sinc^2(\tfrac 1 2 z) - \psi(z)) | \leq C_{10} | \sin(\tfrac 1 2 z) |,
\end{equation*}
which is \cite[Condition (13)]{GriH06}.
Filter Assumptions~\eqref{eq:filterAssumptionGH3}, \eqref{eq:filterAssumptionGH4} and \eqref{eq:filterAssumptionGH5} yield
\begin{equation*}
	| (\sinc(z) - \chi(z)) | \leq C_{13} | z \sin(\tfrac 1 2 z) |
\end{equation*}
for $\chi = \phi, \psi_0, \psi_1$, which is \cite[Condition (14)]{GriH06}.
\cite[Condition (11) to (14)]{GriH06} are sufficient for the second order estimate of the solution (without the derivative) in \cite[Theorem 1]{GriH06}, which is all we need.

Our proposed filter choice \eqref{eq:filterChoiceNew} in addition to the filter conditions~\eqref{eq:filterAssumptionsE} also fulfill the new 
filter conditions~\eqref{eq:filterAssumptionEGH}, \eqref{eq:filterAssumptionGH4} holds true with $C_{12} = 0$. This implies that scheme~\eqref{eq:timeIntegrationScheme} with $\psi_B \equiv \phi_B \equiv 1$ and \eqref{eq:filterChoiceNew} is of second order also for multiple high frequencies in $\OOmega$.


\begin{remark}
\cite[Theorem 1]{GriH06} of Grimm and Hochbruck requires the non-linearity and its derivatives
$g$, $g_y$ and $g_{yy}$ to be bounded globally. This would exclude our $g(\EE)= \GG \EE$, which is linear and thus unbounded.
An inspection of the proof however reveals that $g$ has only to be bounded on the solution $\EE(t)$ and
on $\phi \EE(t)$, such that it is sufficient that $g$ is bounded on a ball.
\end{remark}

\section{Numerical experiments}
\subsection{Laser plasma interaction -- triple splitting}
\label{sec:numexp}
As illustration of the convergence result we setup an experiment as also shown in \cite{TuePLH10}.
The settings are taken from the thin foil experiment above in Section \ref{sec:introPhysics}.

We use the laser pulse from \eqref{eq:laserPulse}
as initial value for the fields and zero initial impulses.
In vacuum, this models a laser pulse propagating only in $x$-direction.
We assume a domain which is homogeneous in $y$ and $z$ direction such that the continous equations \eqref{eq:maxwellsEquationsLaserReflect} simplify to
\begin{subequations}
\label{eq:maxwellsEquations1d}
\begin{align}
	 \partial_t \pp_y & = \EE_y, \\
	 \partial_t \EE_y & = -\partial_x \BB_z - f^2\rho \pp_y, \\
	 \label{eq:maxwellsEquations1dBupdate}
	 \partial_t \BB_z & = -\partial_x \EE_y
\end{align}
\end{subequations}
with periodic boundary conditions.
The density profile is chosen as
\begin{equation}
  \label{eq:rhoLaserReflect}
  \rho(x) =
  \begin{cases}
    \rho_{F},\quad & \text{if} ~ x \in F, \\
    0, \quad & \text{otherwise},
  \end{cases}
\end{equation}
where $F$ is the area covered by the foil.
Spatial discretization is done with finite forward differences for the space dertivative of $\BB$-field and backwards differences for the $\EE$-field.
This corresponds to the Yee grid to the one-dimensional situation.
Assumption~\ref{ass:physicalAndDiscretizationAssumptions} is satisfied.
The bounds of Assumption~\ref{ass:initialData} are also statisfied, exploiting that $|\EE|$ and $|\BB|$ are smaller than machine precision and thus the error of setting them to zero in the foil $F$ is not larger then the round-off error when evaluating the exponential function numerically.

We show the error in $\EE$, $\pp$ and $\BB$. The error in $\EE$ dominates the error in $\BB$ by 
almost one magnitude. The error in the impulses $\pp$ almost coincides with the error in the electric field $\EE$ if no filters are used. If the filter choice~\eqref{eq:filterChoiceJalo} is employed
the error in $\EE$ and $\pp$ coincide away from even multiples of $\pi$. Thus the peaks in this
case are in the error of $\EE$ only.
The left column in Fig.~\ref{fig:ts} shows the error of the method for $\rho_F = 64\cdot 10^8, \omega = 8\cdot 10^4$ and the right column corresponds to with $\rho_F = 9\cdot 10^{6}, \omega = 3\cdot 10^{3}$.
We show the euclidean norm of the absolute error at $T=20$ versus step size $\tau$ for
the numerical solution of~\eqref{eq:timeIntegrationScheme} measured against 
the spatially discrete reference solution~\eqref{eq:maxwellsEquationsDiscrete} calculated
with the \texttt{expmv} routine from~\cite{MohH11}. In the upper row no filter functions were used, resulting in large broad error peaks. In the middle row the filter choice~\eqref{eq:filterChoiceJalo} results in very sharp error peaks around even multiples of $2\pi/\omega$. 
As predicted by our theory the bottom row shows second order convergence independent of $\omega$. 
For the zoom the range of step sizes is $\tau \in [0.923 \cdot 2\pi/\omega, 1.075 \cdot 2\pi/\omega]$ if no filter function is used and it is much smaller if a filter function is chosen, i.e $\tau \in [0.997 \cdot 2\pi/\omega, 1.003 \cdot 2\pi/\omega]$.
\begin{figure}[tbh]
  \centering
  \includegraphics[width=0.49\textwidth]{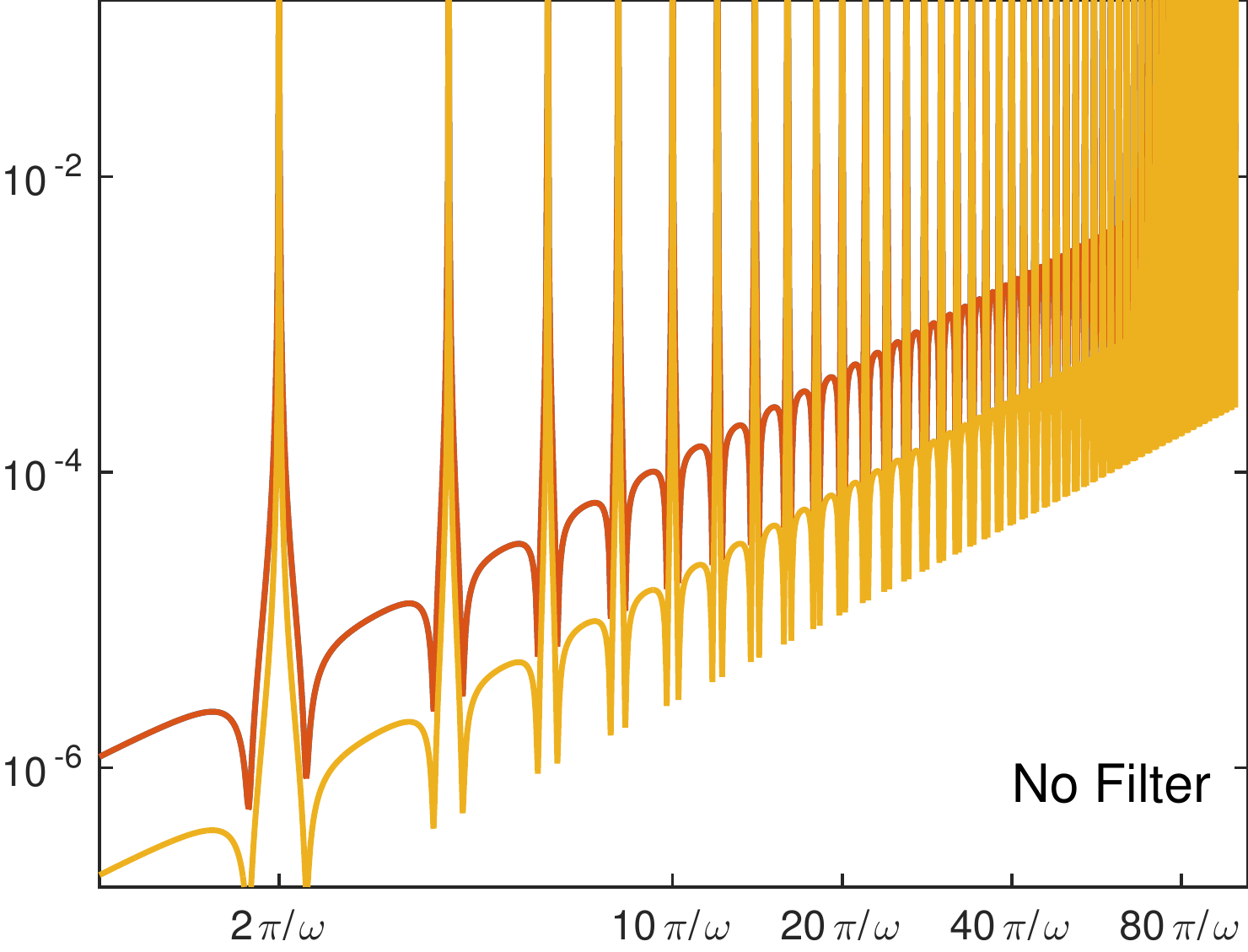}
  \includegraphics[width=0.49\textwidth]{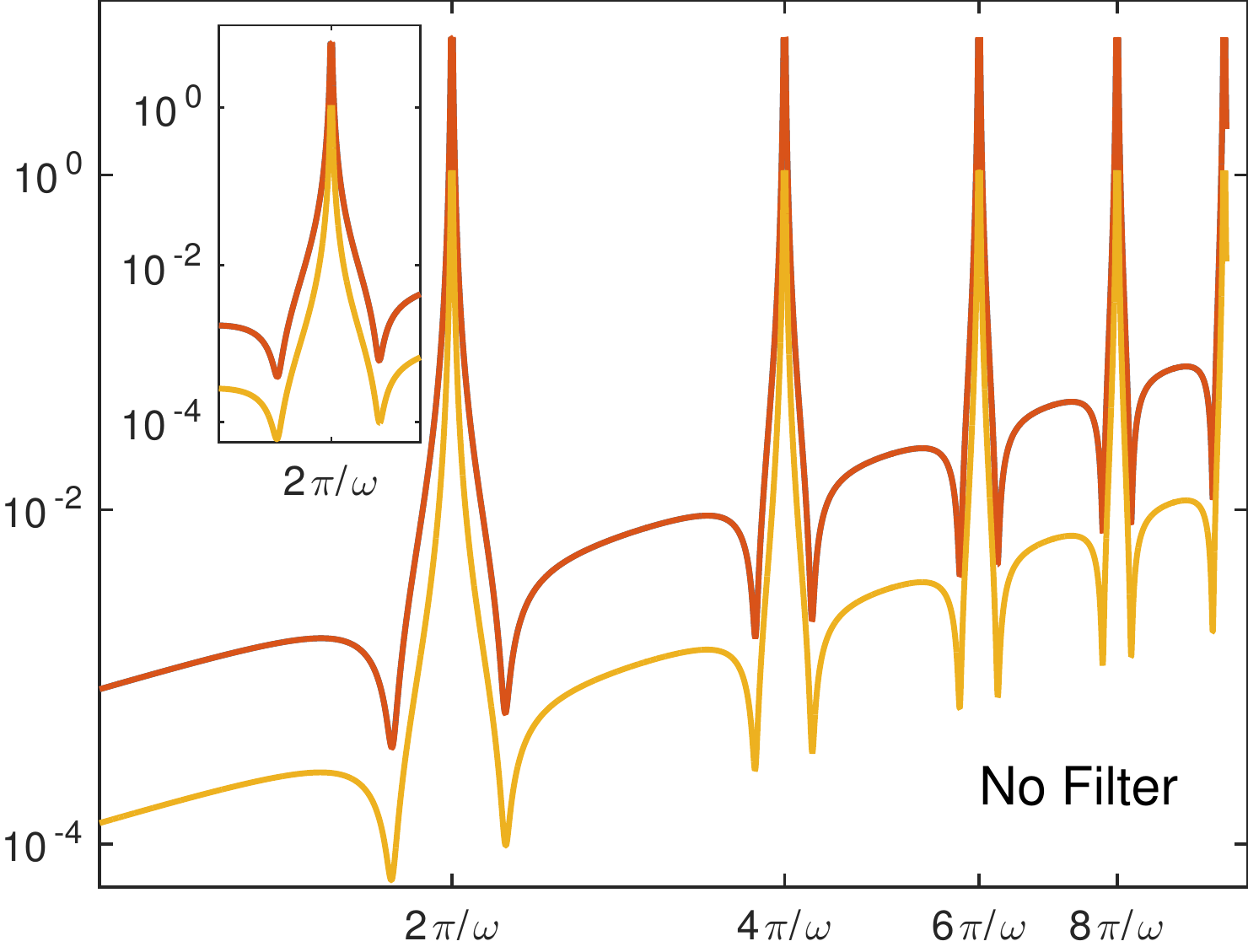}
  \\
  \includegraphics[width=0.49\textwidth]{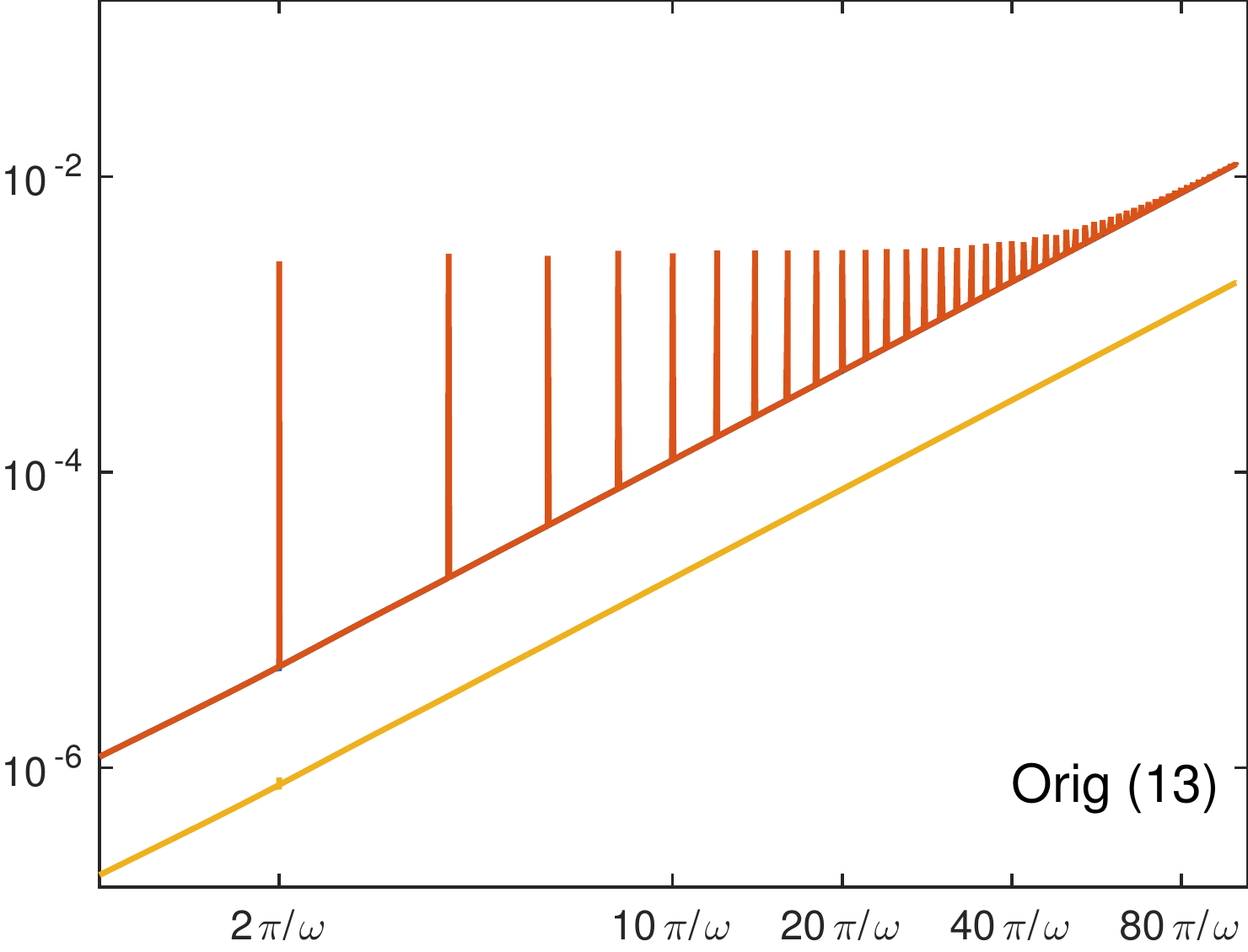}
  \includegraphics[width=0.49\textwidth]{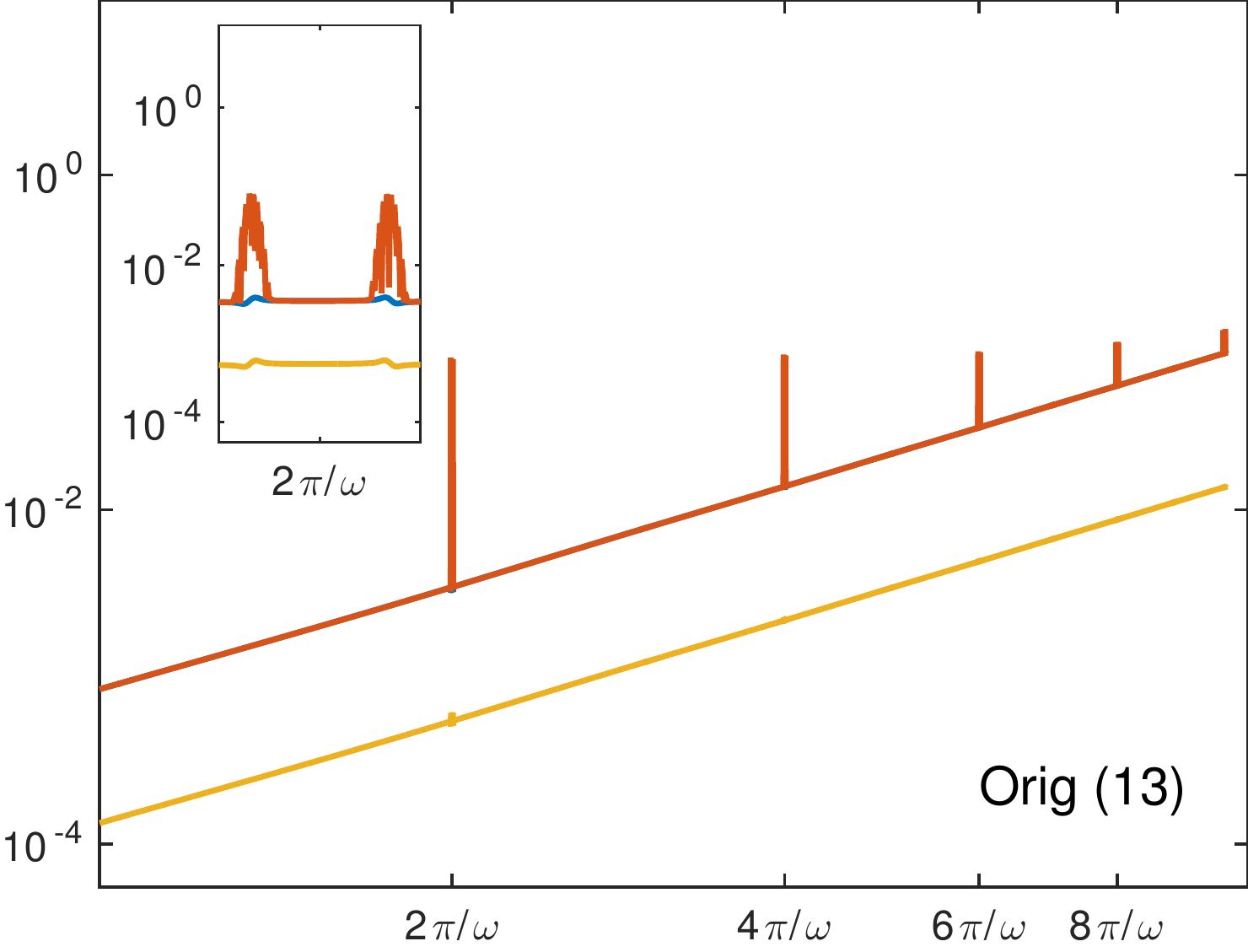}
  \\
  \includegraphics[width=0.49\textwidth]{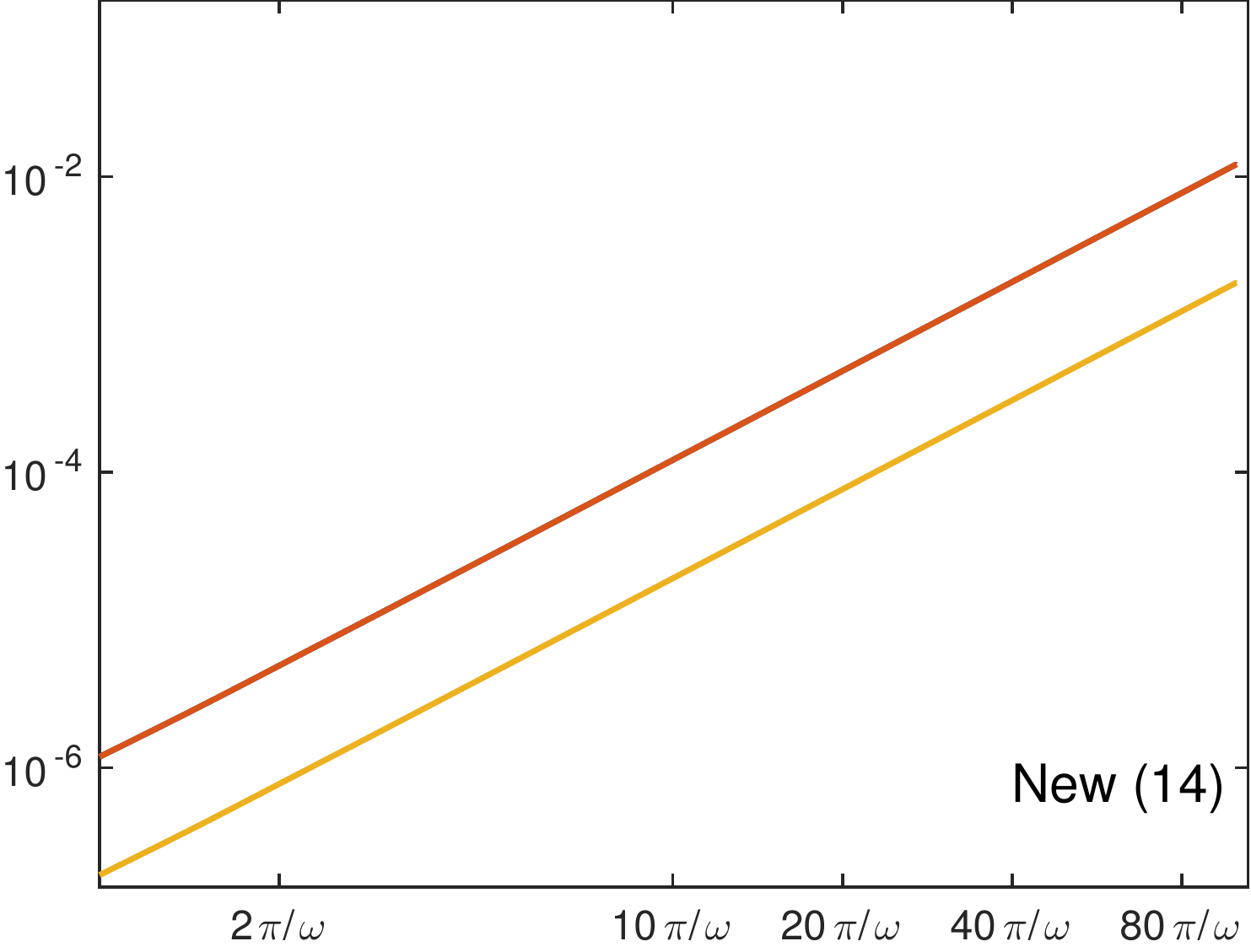}
  \includegraphics[width=0.49\textwidth]{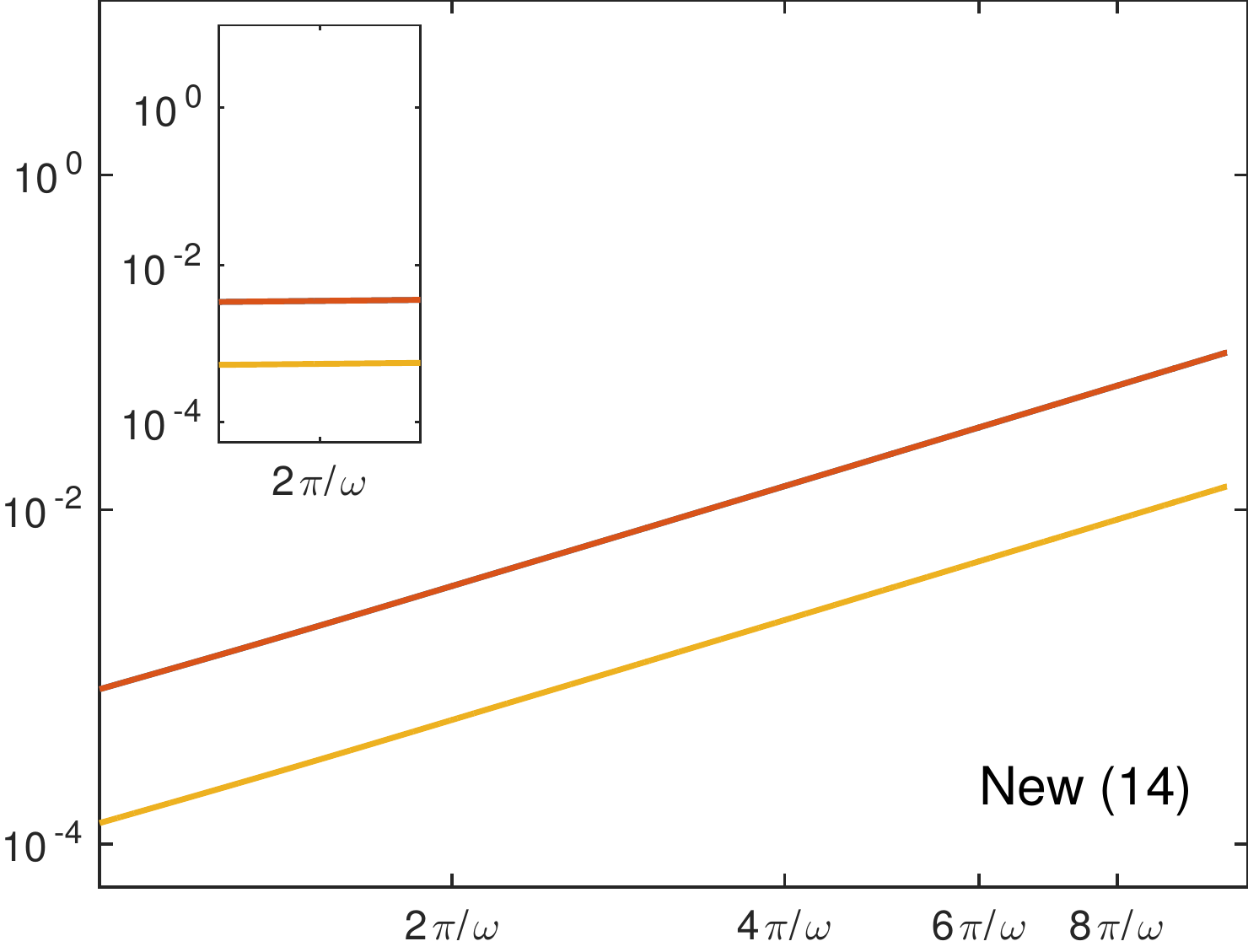}
  \caption{Euclidean norm of the absolute error of $\EE$, $\pp$ and $\BB$
    at $T=20$ versus step size for
    the numerical solution~\eqref{eq:timeIntegrationScheme} measured against 
    the spatially discrete reference solution of~\eqref{eq:maxwellsEquationsDiscrete}.
    Left: $\omega = 8\cdot 10^4$; Right: $\omega = 3\cdot 10^{3}$; Top: No filter;
    Middle: Filter of~\cite{TuePLH10}; Bottom: New filter choice~\eqref{eq:filterChoiceNew}. }
  \label{fig:ts}

\end{figure}

\subsection{Klein-Gordon type equation -- two step method}
\label{sec:hlw}
We consider a one-dimensional Klein-Gordon type equation for one component of 
the electric field with periodic boundary conditions on the interval $[-10, 14]$, where 
the plasma occupies the region $(10,11)$. This equation is obtained by eliminating $\BB$ and $\pp$ from~\eqref{eq:maxwellsEquationsLaserReflect}.
Discretization in space is by symmetric second order finite differences on the equidistant 
grid $x$ with grid points $x_j = -10+j h$, $j=1,\dots,N$, with $N=240$ and spacing 
$h = 24/N$. The initial value $\EE_0$ is given by~\eqref{eq:laserPulse} ($\bar{x}=0$, $\sigma_0=10$) 
evaluated on the grid and initial velocity by 
$(\dot{\EE}_0)_j =  ( (\frac{2 \pi}{\sigma_0})^2 x_j  \cos(2\pi x_j)+ 2\pi \sin(2\pi x_j) )\mathrm{e}^{-\frac{2\pi^2}{\sigma_0^2} x_j^2}$.
That is we solve for $\EE(t)$
\begin{equation}
  \label{eq:KG}
  \begin{aligned}
    \partial_{tt}\EE(t) &= \GG \EE(t) - \OOmega \EE(t) \mbox{ for } t \in [0,3]
    \\
    \EE(0) &= \EE_0, \quad \partial_t \EE(0) = \dot{\EE}_0 
  \end{aligned}
\end{equation}
with, using Matlab notation, matrices
$ \GG = \mathrm{spdiags}([e, -2*e, e],-1:1,N,N)/h^2;$ $\GG(1,N) = 1/h^2;$ $\GG(N,1) = 1/h^2;$ for 
a vector $e$ with all ones and $f = 0*x;$ $f(x<11\ \&\ x>10) = 1;$ $\OOmega = \mathrm{diag}(\omega*f)$.
with $\omega = 9\cdot 10^3$. 

We have implemented the two step method from~\cite[XIII.2.2]{HaiLW06}
with even real-values filter functions $\psi$ and $\phi$, with $\psi(0)=\phi(0)=1$.
\begin{align*}
  &\mathrm{(A)} &\ \psi(z) &= \sinc^2(\tfrac 1 2 z) & \phi(z) &= 1 & \mbox{ Gautschi  \cite{Gautschi61}}
  \\
  &\mathrm{(B)} &\ \psi(z) &= \sinc(z) & \phi(z) &= 1 & \mbox{ Deuflhard  \cite{Deuflhard79}}
  \\
  &\mathrm{(C)} &\ \psi(z) &= \sinc(\tfrac 1 2 z) \phi(z) & \phi(z) &= \sinc(z) & \mbox{ Garcia-Archila et al. \cite{Garcia-Archilla1998}}
  \\
  &\mathrm{(D)} &\ \psi(z) &= \sinc^2(\tfrac 1 2 z) & \phi(z) & \mbox{ from } \eqref{eq:phiHocL} \mbox{ below } & \mbox{ Hochbruck, Lubich \cite{HocL99}}
  \\
  &\mathrm{(E)} &\ \psi(z) &=  \sinc^2(z) & \phi(z) &= 1 & \mbox{ Hairer, Lubich  \cite{Hairer2001}}
  \\
  &\mathrm{(F)} &\ \psi(z) &= \eta(z)\sinc(z)^2 & \phi(z) &= \sinc(\tfrac 1 2 z) & \eqref{eq:filterChoiceNew}
  \\
  &\mathrm{(G)} &\ \psi(z) &= \eta(z)\sinc(z) & \phi(z) &= \sinc(\tfrac 1 2 z) &  
  \eqref{eq:filterChoiceJalo}~\mbox{ Liljo et al.  \cite{TuePLH10}}
  \\
  &\mathrm{(H)} &\ \psi(z) &= \sinc(\tfrac 1  2 z) & \phi(z) &= \sinc(z) & 
  \\
  &\mathrm{(I)} &\ \psi(z) &=  \sinc(z) & \phi(z) &=\sinc(\tfrac 1 2 z), & 
\end{align*}
where 
\begin{equation}
	\label{eq:phiHocL}
	\phi(z) = \sinc(z)(1 + \tfrac 1 3 \sin^2(\tfrac 1 2 z))
\end{equation}
in method (D). The alphabetic labels for methods (A) - (E) follow the convention of~\cite{HaiLW06}.
Method (F) corresponds to our choice, \eqref{eq:filterChoiceNew}, with 
$\eta(z) := \tfrac 1 2 (1+\cos(z))$ coming for free from the triple splitting. 
Method (G) corresponds to the choice~\eqref{eq:filterChoiceJalo} considered in~\cite{TuePLH10,Lil10}.

Figure~\ref{fig:hlw} shows the norm of the absolute error in 
euclidean norm versus the step size.
For this linear test problem method (E) shows the same behavior as (A), 
the behavior of (D) is similar to (C) and (H) is similar (I), therefore results
for (E), (D) and (H) are not displayed. The inset is a zoom to step sizes
around $2\pi/\omega$ showing the error for  
$\tau \in [(2\pi-5\cdot 10^{-3})/\omega, (2\pi+5\cdot 10^{-3})/\omega]$. 
For this linear test problem one observes second order convergence as soon as there
is a double zero of $\psi$ at even multiples of $2\pi$. The condition on $\phi$ 
seems to be less important. However comparing (I) and the method (B) 
it is observed that the resonance peak is much sharper for (I), reflecting the influence of
$\phi$ in this test problem. Though the filter functions
of methods (G), (H) and (I) satisfy the assumptions for first order convergence uniformly 
in $h\omega$ as predicted by~\cite[Theorem XIII.4.1]{HaiLW06}, c.f. Remark~\ref{rem.2}, sharp 
resonance peaks are observed. Currently we suspect a mistake in the proof of the
Theorem XIII.4.1 there.

\begin{figure}
  \centering
  \includegraphics[width=0.49\textwidth]{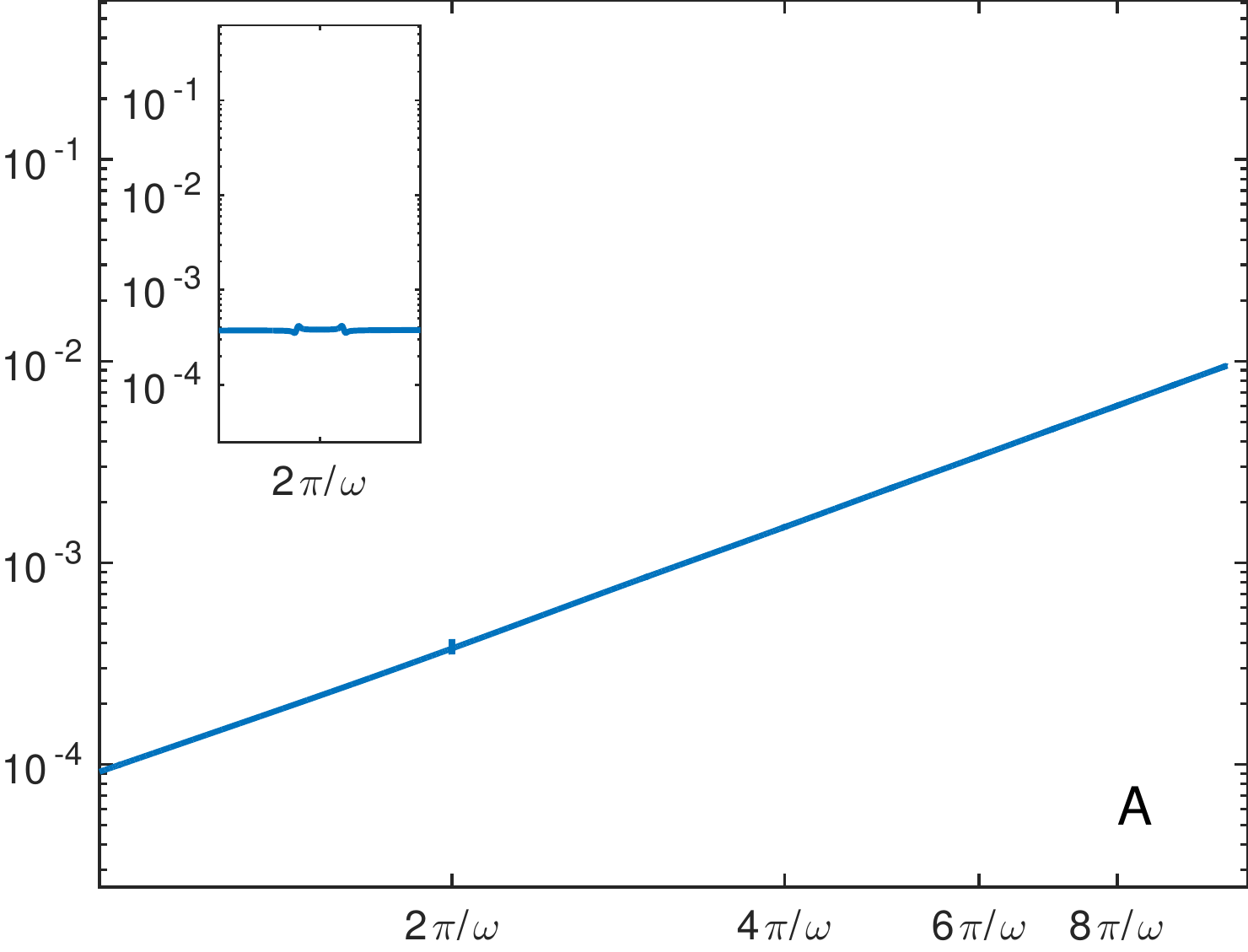}
  \includegraphics[width=0.49\textwidth]{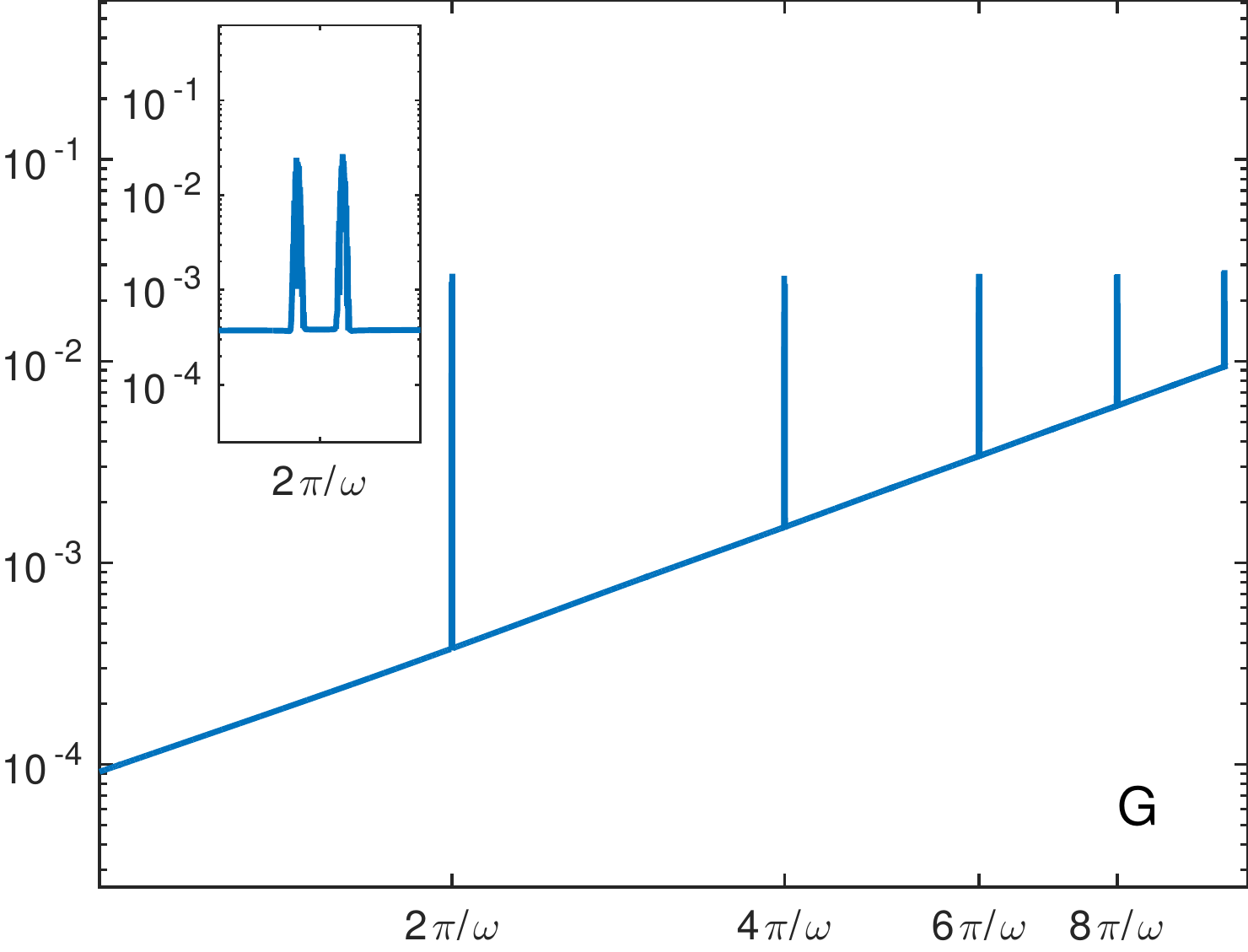}
  \\
  \includegraphics[width=0.49\textwidth]{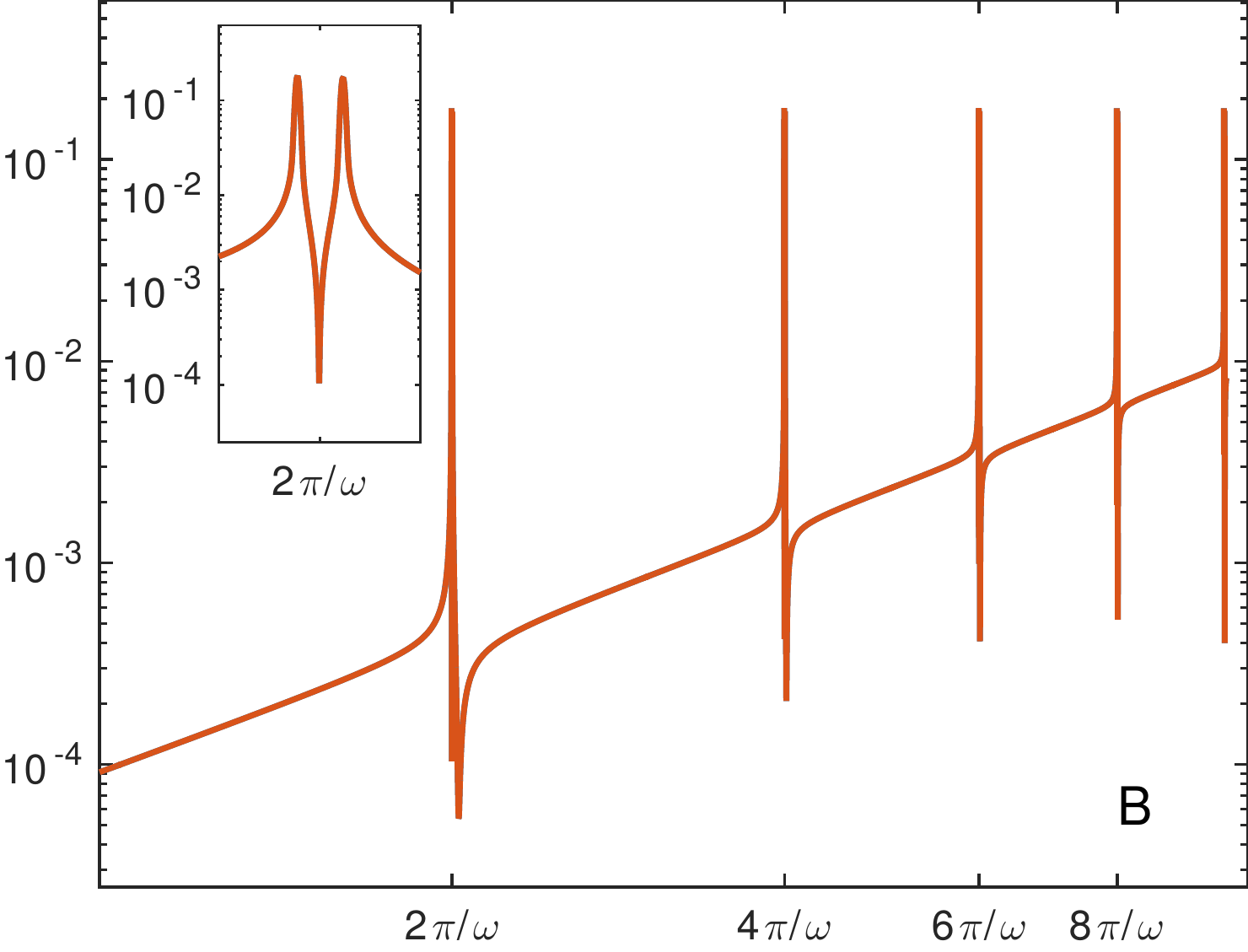}
  \includegraphics[width=0.49\textwidth]{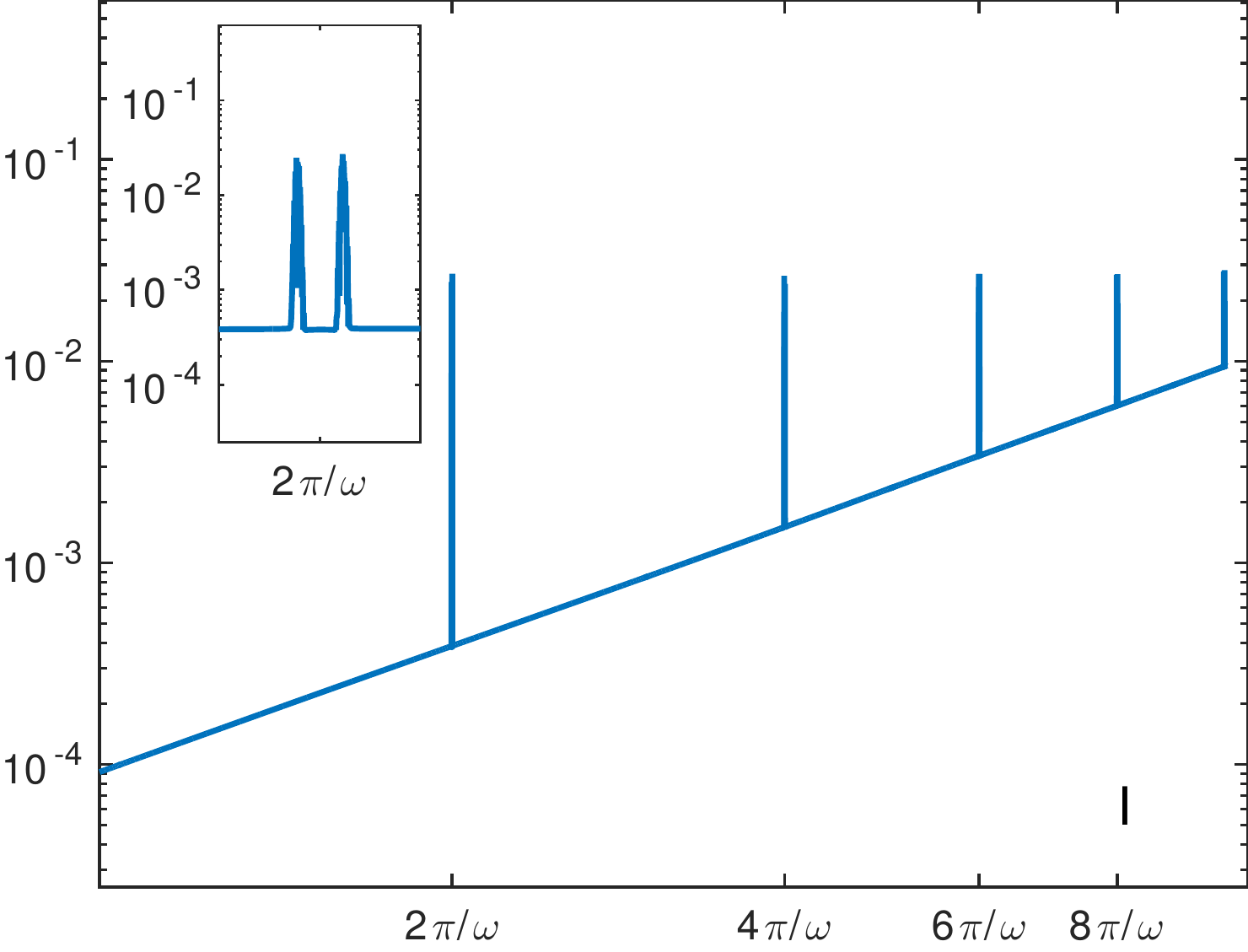}
  \\
  \includegraphics[width=0.49\textwidth]{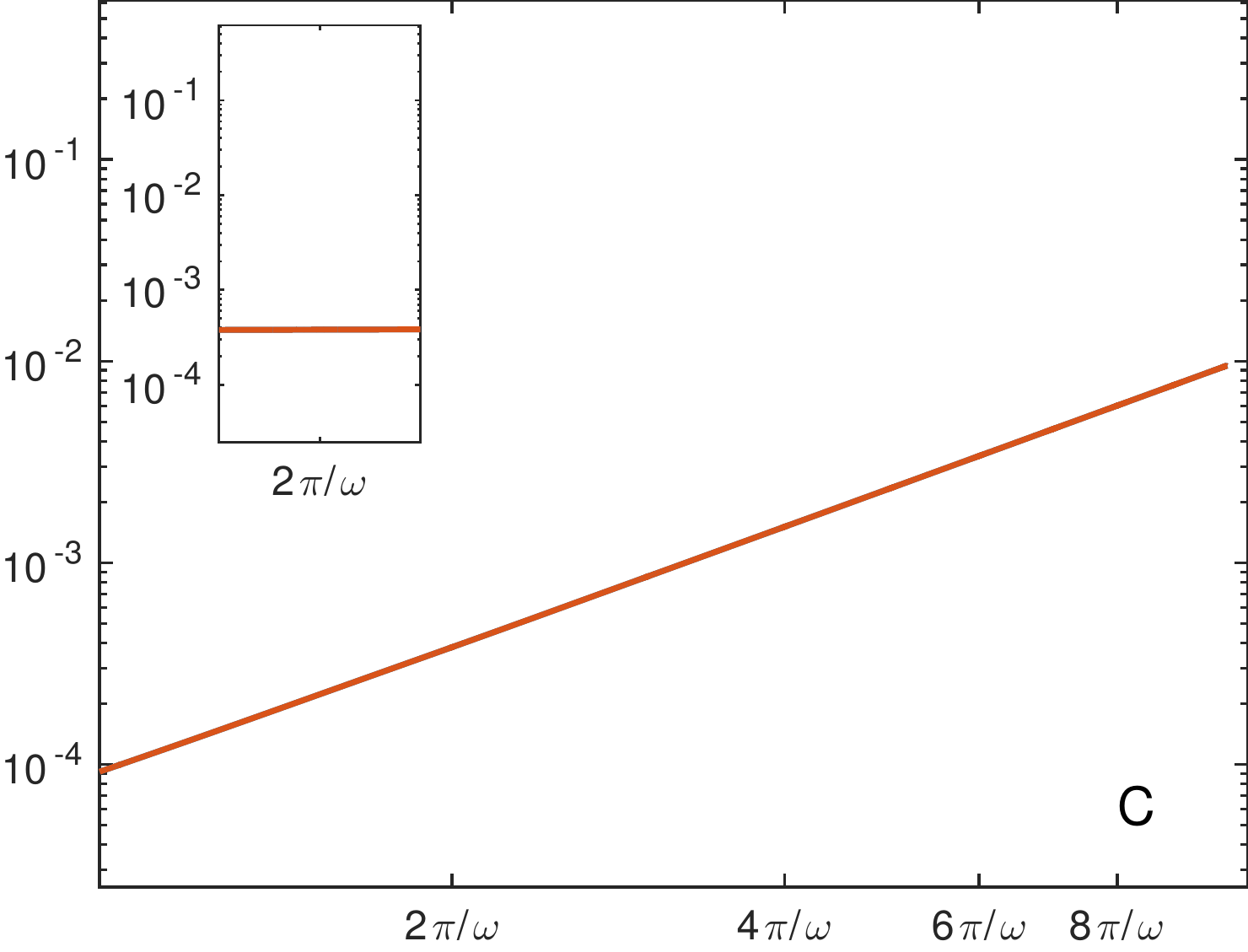}
  \includegraphics[width=0.49\textwidth]{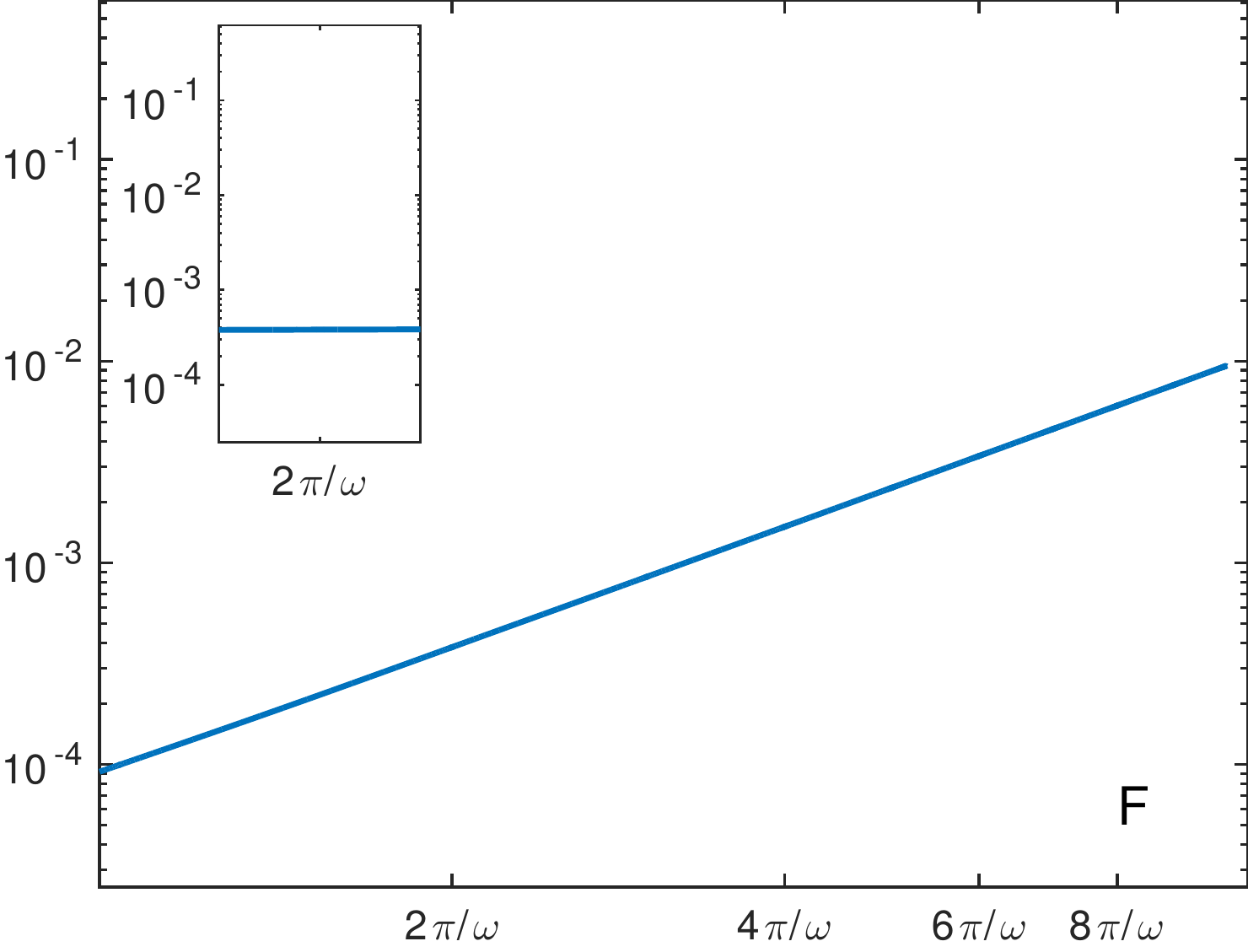}
  \caption{Euclidean norm of the absolute error at $T=3$ versus step size for the Klein-Gordon type equation
    \eqref{eq:KG}. Results are shown
    for filter choices (A), (B), (C), (G), (I) and (F). The inset shows a zoom to step-sizes
  in the interval $ [(2\pi-5\cdot 10^{-3})/\omega, (2\pi+5\cdot 10^{-3})/\omega]$. }
  \label{fig:hlw}
\end{figure}
\section*{Acknowledgments}
We thank Ernst Hairer and Ludwig Gauckler for answering all our questions when 
we were carefully checking the proof of Theorem~\ref{thm:hlwAdoption} and the relevant 
part of~\cite[Theorem XIII.4.1]{HaiLW06}. We thank
Marlis Hochbruck and Volker Grimm for their interest in our work and 
stimulating discussions.

\bibliographystyle{plain}
\bibliography{main}

\end{document}